\documentclass[11pt]{amsart}

\usepackage{fullpage}  

\usepackage[pagewise,mathlines]{lineno} 

\usepackage{etoolbox} 

\newcommand*\linenomathpatch[1]{%
	\cspreto{#1}{\linenomath}%
	\cspreto{#1*}{\linenomath}%
	\csappto{end#1}{\endlinenomath}%
	\csappto{end#1*}{\endlinenomath}%
}
\newcommand*\linenomathpatchAMS[1]{%
	\cspreto{#1}{\linenomathAMS}%
	\cspreto{#1*}{\linenomathAMS}%
	\csappto{end#1}{\endlinenomath}%
	\csappto{end#1*}{\endlinenomath}%
}

\expandafter\ifx\linenomath\linenomathWithnumbers
\let\linenomathAMS\linenomathWithnumbers
\patchcmd\linenomathAMS{\advance\postdisplaypenalty\linenopenalty}{}{}{}
\else
\let\linenomathAMS\linenomathNonumbers
\fi

\linenomathpatch{equation}
\linenomathpatchAMS{gather}
\linenomathpatchAMS{multline}
\linenomathpatchAMS{align}
\linenomathpatchAMS{alignat}
\linenomathpatchAMS{flalign}

\usepackage[utf8]{inputenc}

\usepackage{amsmath}  
\usepackage{amsthm}  
\usepackage{amssymb}  
\usepackage{bm}  
\usepackage{mathtools}  

\usepackage{enumerate}  

\usepackage{graphicx}  
\usepackage[dvipsnames]{xcolor}  
\usepackage[small,bf]{caption}  

\graphicspath{{figures/}}  
\graphicspath{{../figures/}}  

\usepackage{breakcites}  
\usepackage{url}  
\usepackage{hyperref}  

\hypersetup{colorlinks, citecolor=black, filecolor=black,
	linkcolor=black, urlcolor=black}

\usepackage{algorithm}
\usepackage[noend]{algpseudocode}

\usepackage{comment}  

\theoremstyle{plain}

\newtheorem{outtheorem}{Theorem}

\newtheorem{theorem}{Theorem}[section]
\newtheorem{proposition}[theorem]{Proposition}
\newtheorem{lemma}[theorem]{Lemma}
\newtheorem{corollary}[theorem]{Corollary}

\theoremstyle{definition}

\newtheorem{definition}[theorem]{Definition}
\newtheorem{notation}[theorem]{Notation}
\newtheorem{exampleEnv}[theorem]{Example}
{\pushQED{\qed}\exampleEnv}
{\popQED\endexampleEnv}
\newenvironment{example*}{\exampleEnv}{\endexampleEnv}
\newtheorem{remarkEnv}[theorem]{Remark}
{\pushQED{\qed}\remarkEnv}
{\popQED\endremarkEnv}
\newenvironment{remark*}{\remarkEnv}{\endremarkEnv}

\newtheorem*{solutionEnv}{Solution}
\ifdefined\showsolution
\newenvironment{solution}
{\pushQED{\qed}\solutionEnv}
{\popQED\endsolutionEnv}
\newenvironment{solution*}{\solutionEnv}{\endsolutionEnv}
\else  

\excludecomment{solution}
\newenvironment{solution*}{}{}
\excludecomment{solution*}
\fi

\ifdefined\showcomment 
\newtheorem*{comment}{\normalfont\emph{Comment}}
\else  
\newenvironment{comment}{}{}
\excludecomment{comment}
\fi
\newcommand{\tr}{\mathop\mathrm{tr}}  

\def\<{\left\langle}  
\def\>{\right\rangle}

\usepackage[T1]{fontenc}

\def\a{\alpha}
\def\b{\beta}

\def\gg{\mathfrak{g}}
\def\hh{\mathfrak{h}}
\def\aa{\mathfrak{a}}

\def\mm{\mathfrak{m}}
\def\nn{\mathfrak{n}}

\def\kk{\mathfrak{k}}
\def\ll{\mathfrak{l}}

\def\zz{\mathfrak{z}}

\def\uu{\mathfrak{u}}

\def\gl{\mathfrak{gl}}

\DeclareMathOperator{\Ric}{Ric}
\DeclareMathOperator{\ric}{ric}
\DeclareMathOperator{\Ad}{Ad}
\DeclareMathOperator{\ad}{ad}

\DeclareMathOperator{\Aut}{Aut}
\DeclareMathOperator{\End}{End}
\DeclareMathOperator{\pr}{pr}
\DeclareMathOperator{\SO}{SO}
\DeclareMathOperator{\Or}{O}
\DeclareMathOperator{\SU}{SU}
\DeclareMathOperator{\GL}{GL}

\DeclareMathOperator{\Sym}{Sym}
\DeclareMathOperator{\Lie}{Lie}
\DeclareMathOperator{\Der}{Der}
\DeclareMathOperator{\Id}{Id}
\DeclareMathOperator{\M}{M}
\DeclareMathOperator{\Bk}{B}
\DeclareMathOperator{\Hm}{H}
\DeclareMathOperator{\h}{h}
\DeclareMathOperator{\mo}{m}

\DeclareMathOperator{\Rscal}{R}

\DeclareMathOperator{\J}{J}

\newcommand{\ho}{\mathcal{H}}

\usepackage{tikz-cd}
\usepackage{graphicx}
\graphicspath{ {./images/} }

\numberwithin{equation}{section}

\title{\bf Finite extinction time of a family of homogeneous Ricci flows}
\author{Roberto Araujo}
\address{Institute of Mathematics, Polish Academy of Sciences\\
	Jana i Jędrzeja \' Sniadeckich 8\\
	00-656 Warsaw\\
	Poland}
\email{raraujo@impan.pl}
\thanks{This project was partially done while the author was funded by the Deutsche Forschungsgemeinschaft (DFG, German Research Foundation) under Germany’s Excellence Strategy EXC 2044–390685587, Mathematics Münster: Dynamics–Geometry–Structure and the CRC 1442 Geometry: Deformations and Rigidity}

\begin{document}
	
	\begin{abstract}
	We show that for a broad family of noncompact homogeneous Riemannian manifolds, the corresponding homogeneous Ricci flow solutions have finite extinction time, thereby confirming the dynamical Alekseevskii conjecture for these spaces. As an application, we prove that on such homogeneous manifolds $G/H$, the space of all $G$-invariant positive scalar curvature metrics is contractible.
	\end{abstract}

\maketitle
	
	
	\section{Introduction}
	The Ricci flow is the geometric evolution equation on a smooth manifold $M$ given by 
	\begin{equation}\label{eq:general RF}
		\frac{\partial g(t)}{\partial t} = -2\ric(g(t)) , \hspace{0.5cm} g(0)=g_0,
	\end{equation}
	where $\text{ric}(g)$ is the Ricci $(0,2)$-tensor of the Riemannian manifold $(M,g)$.
	
	Hamilton introduced the Ricci flow in \cite{ham} and proved short time existence and uniqueness when $M$ is compact. 
	A maximal Ricci flow solution $g(t)$, $t \in[0,T)$, is called \textit{immortal} if $T=+\infty$; otherwise we say that the flow has \textit{finite extinction time}.
	
	A Riemannian manifold $(M,g)$ is called \textit{homogeneous} if its isometry group acts transitively on it. By the uniqueness of Ricci flow solutions it follows immediately that the isometries are preserved along the flow, thus a solution $g(t)$ from a homogeneous initial metric $g_0$ remains homogeneous for the same isometric action. In this setting, the Ricci flow equation reduces to an autonomous nonlinear ordinary differential equation, which we refer to as the \textit{homogeneous Ricci flow}.
	
	Lafuente has shown in  \cite{laf} that  a homogeneous Ricci flow solution has finite extinction time if and only if the scalar curvature blows up in finite time, or equivalently, if and only if the scalar curvature eventually becomes positive along the flow. Moreover, Bérard-Bergery has shown in \cite{ber} that a manifold admits a homogeneous Riemannian metric of positive scalar curvature if and only if its universal cover is not diffeomorphic to $\mathbb{R}^n$. 
	B\"ohm and Lafuente then questioned in \cite{bl18} whether the converse is also true, namely they asked whether the universal cover of an immortal homogeneous Ricci flow solution is always diffeomorphic to $\mathbb{R}^n$. This got established later as the \textit{dynamical Alekseevskii conjecture} \cite{mfo22}.
	
Böhm had previously shown the conjecture to be true for coverings of compact homogeneous manifolds \cite[Theorem 3.2]{boe}. In \cite{ara1}, it was shown that the dynamical Alekseevskii conjecture is confirmed in the case where the isometry group of the homogeneous Riemannian manifold is, up to a covering, a Lie group product with a compact semisimple factor. Moreover, in \cite{ara2}, it was shown that awesome homogeneous metrics on manifolds with a transitive action by a semisimple Lie group also satisfy the dynamical Alekseevskii conjecture. The family of awesome metrics represents a big source of interesting examples for the conjecture, since in \cite{dl} and \cite{dlm} the authors have shown for all but a finite collection of noncompact simple Lie groups that they admit Ricci negative awesome left-invariant metrics. Hence, for those spaces such that the universal cover is not diffeomorphic to $\mathbb{R}^n$ the confirmation of the conjecture implies a change of regime of the Ricci flow: from one in which the metric expands in all directions to one such that in some direction the metric shrinks and collapses in finite time.

	From that standpoint, a similarly interesting testing ground is provided by the examples of nonsolvable, nonsemisimple Lie groups first introduced by Will in \cite{will17} and further extended by E. Lauret and Will (see \cite{will} and \cite{lw}) which admit Ricci negative left-invariant metrics. These examples belong to a special family of left-invariant metrics that satisfy some compatibility conditions between the Lie algebraic structure and the geometry. The simplest example in this family is $M=G = \left(\SU(2)\times \mathbb{R}\right)\ltimes \mathbb{R}^3$, where $\mathbb{R}^3$ is an irreducible $\SU(2)$-representation, $\mathbb{R}$ acts on  $\mathbb{R}^3$ as a multiple of the identity, and the left-invariant metric satisfies $\left(\SU(2)\times \mathbb{R}\right)\perp \mathbb{R}^3$. E. Lauret and Will have shown in \cite{lw} that any compact semisimple Lie group can arise as the Levi factor of a Lie group admitting a Ricci negative metric.
	
	In this article, we introduce a family of homogeneous Riemannian manifolds that includes, as a very particular case, the construction in \cite[Theorem 5.1]{lw}. We then show that this family of metrics is Ricci flow invariant and satisfies the dynamical Alekseevskii conjecture. Namely, we prove the following theorem.
	\begin{outtheorem}\label{thm:A}
		Let $G= U\ltimes V$ be a Lie group with Lie algebra $\gg = \uu \ltimes_{\theta} V$, where $V$ is an abelian ideal, $\uu$ is a compact Lie algebra acting on $V$, and the representation $\theta \colon \uu \to \gl(V)$ is via semisimple operators. Let $H$ be a compact subgroup of $U$ and let us consider the family of $G$-invariant metrics on $M=G/H$ such that 
		\begin{enumerate}
			\item $U/H \perp V$; \label{item:standard}
			\item There is an orthogonal weight space decomposition of $V$ into $\uu$-submodules, 
			\[V=\bigoplus_{\a \in \mathcal{I}^*} V^\a,\]
			where $\mathcal{I}^* \subset \uu^*$ and such that for all $Z\in\uu$, $\theta(Z)-\a(Z)\Id$ restricted to $V^\a$ is a semisimple operator with purely imaginary eigenvalues. \label{item:awesome}
		\end{enumerate}		
		Then this family is Ricci flow invariant. Moreover, if the universal cover of $M$ is not diffeomorphic to $\mathbb{R}^n$, then for any initial metric on this family the Ricci flow has finite extinction time.
	\end{outtheorem}
	
	Since we can assume without loss of generality that the presentation $G/H$ is almost-effective, the condition $H \subset U$ is not actually restrictive. 
	
	Note that when $\theta\colon \uu \to \gl(V)$ is injective, $V$ is the abelian nilradical of $\gg$ and condition \eqref{item:standard} is the same as $M$ being $G$-standard in the sense of \cite[Definition 2.1]{bl22}. Moreover, the weight space decomposition $V=\bigoplus_{\a \in \mathcal{I}^*} V^\a$ above always exists as a consequence of the fact that $\uu$ is compact and $\theta(\uu)$ are semisimple operators, since in this case there is an $\Ad(H)$-invariant metric on $V$ such that $\theta(\uu)$ are normal operators. Condition \eqref{item:awesome} is then a compatibility condition between the metric restricted to $V$ and the representation $\theta$.
	We call the family of metrics on $M$ for which there exists a decomposition of $\gg$ satisfying the above hypotheses \textit{$\theta$-adapted standard} $U\ltimes_\theta V$-invariant metrics. 
	
	It is worth remarking that such representations $\theta \in \End(\uu,\gl(V))$ are precisely the ones contained in closed orbits for the conjugation action of $\GL(V)$ on $\End(\uu,\gl(V))$, and as such we call them \textit{stable} representations.
	
	In order to prove Theorem \ref{thm:A}, it seems natural to first get curvature estimates that allow us to understand the behavior of the $\theta$-adapted standard Ricci flow solution $g(t)$ restricted to $V$. However, $G$ may not be unimodular and that yields bad terms in the Ricci tensor with respect to getting these initial curvature estimates, which is indeed related to the existence of Ricci negative metrics on this family. 
	We overcome this difficulty by using the equivalence between the homogeneous Ricci flow and the \textit{unimodular Ricci flow} \cite[Corollary 3.3]{bl18}.
	Indeed, we can straightforwardly get initial unimodular Ricci curvature bounds (see Corollary \ref{cor:Ricci-A and Ricci A-bar}) to be then leveraged throughout the dynamics of the unimodular Ricci flow solution $g^\star(t)$ in order to confirm the conjecture for such spaces.
	More explicitly, using the unimodular Ricci flow we get initial control over the pinching of $\left.g^\star(t)\right\vert_{V\times V}$, which leads to integral curvature bounds for the unimodular Ricci curvature in the direction of $V$. This, in turn, can be used to show that
	eventually a uniform Ricci positive direction emerges along the flow, arising from the intrinsic geometry of $U/H$ in the case where the universal cover of $U/H$ (equivalently of $M$) is not diffeomorphic to $\mathbb{R}^n$.
	This finally yields the finite extinction time of the Ricci flow solution $g(t)$ starting from a $\theta$-adapted standard metric.
	
		An interesting consequence of the truthfulness of the dynamical Alekseevskii conjecture is the contractibility of the space of $G$-invariant positive scalar curvature metrics on a given homogeneous space $G/H$. We finish the paper showing that this is indeed the case for $\Lie(G)= \uu \ltimes_{\theta} V$ as above. 
	\begin{outtheorem}\label{thm:B}
			Let $G=U\ltimes_\theta V$ be a Lie group as in Theorem \ref{thm:A} and let $H$ be a compact subgroup of $U$. Then the space of $G$-invariant positive scalar curvature metrics on $G/H$ is contractible.
	\end{outtheorem}
	The strategy to prove Theorem \ref{thm:B} is to find that there is a scalar curvature nondecreasing path connecting an arbitrary $U\ltimes_\theta V$-invariant metric and the set of $\theta$-adapted standard ones. And from the $\theta$-adapted one we can complete the argument using the Ricci flow and Theorem \ref{thm:A}.  \\
	
		The structure of this article is the following.
		In Section \ref{section:Prelim}, we establish some of the preliminaries on homogeneous Riemannian manifolds and the (unimodular) homogeneous Ricci flow we will need in the following sections.
		In Section \ref{section:main intro section}, we introduce the family of homogeneous Riemannian manifolds we are interested in, and we show the unimodular Ricci flow invariance of the set of $\theta$-adapted standard metrics. In Section \ref{section:algebraic bounds}, we then proceed to establish a priori algebraic bounds that exploit the compatibility of the Lie brackets and the metric in order to derive monotone quantities for our dynamics. In Section \ref{section:main result}, we proceed to the long-time behavior analysis  to prove Theorem \ref{thm:A}. Finally, in Section \ref{section:connectedness} we prove Theorem \ref{thm:B} and conclude by showing the Ricci flow invariance of a slight generalization of the family of $\theta$-adapted standard metrics which account for a nonabelian nilradical. \\
	
	\textit{Acknowledgments.} It is a pleasure to thank my PhD advisor, Christoph B\"ohm, for his support and for generously sharing his knowledge with me; in particular, for sharing with me the argument about how the dynamical Alekseevskii conjecture implies the contractibility of the set of $G$-invariant positive scalar curvature metrics, which we present in Section \ref{section:connectedness}. I would also like to thank Ramiro Lafuente for the useful comments on a first version of this article. And finally, thank Jorge Lauret for the insightful conversations and, in particular, for suggesting that I attempt to extend the family of homogeneous spaces for which our techniques apply, culminating in Theorem \ref{thm:A} as presented here.
	
	\section{Preliminaries on homogeneous Ricci flows} \label{section:Prelim}
	
	A Riemannian manifold $(M, g)$ is said to be homogeneous if its isometry group
	$I(M, g)$ acts transitively on $M$. If $M$ is connected (which we assume from here onward unless otherwise stated), then each transitive, closed Lie subgroup $G <I(M, g)$ gives
	rise to a presentation of $(M, g)$ as a homogeneous space with a $G$-invariant metric
	$(G/H, g)$, where $H$ is the isotropy subgroup of $G$ fixing some point $p \in M$. We call this space a \textit{homogeneous Riemannian manifold}. 
	
	The $G$-action induces a Lie algebra homomorphism $\gg \to \mathfrak{X}(M)$ assigning to each $X \in \gg$ a Killing field on $(M, g)$, also denoted by $X$, and given by
	\begin{equation}\label{eq:infinitesimal correspondence}
		X(q) \coloneqq \left. \frac{d}{dt}\right|_{t=0} \exp (tX)\cdot q, \hspace{0.5cm} q \in M.
	\end{equation}
	
	If $\hh$ is the Lie algebra of the isotropy subgroup $H < G$ fixing $p \in M$, then it can be characterized as those $X \in \gg$ such that $X(p)=0$. Given that, we can take a complementary Ad($H$)-module $\mm$ to $\hh$ in $\gg$ and identify $\mm \cong T_pM$ via the infinitesimal correspondence \eqref{eq:infinitesimal correspondence}.
	
	In general, a homogeneous space $G/H$ is called \textit{reductive} if there exists a complementary vector space $\mm$ such that for the Lie algebras of $G$ and $H$, respectively $\gg$ and $\hh$, we have
	\begin{align*}
		\gg = \hh \oplus \mm \nonumber, \hspace{0.5cm} \Ad(H)(\mm) \subset \mm.
	\end{align*}
	
	This is always possible in the case of homogeneous Riemannian manifolds since, by a classic result in Riemannian geometry \cite[Chapter VIII, Lemma 4.2]{doC}, an isometry is uniquely determined by the image of a point $p$ and its derivative at $p$. Hence, the isotropy subgroup $H$ is a closed subgroup of $\SO(T_pM)$ and thus compact. Since $H$ is compact, one can average over an arbitrary inner product on $\gg$ to make it Ad($H$)-invariant and hence take $\mm \coloneqq \hh^{\perp}$. With this choice, one can identify $\mm \cong T_{eH}G/H$, and under this identification there is a one-to-one correspondence between $G$-invariant metrics in $M \coloneqq G/H$, with $p \cong eH$, and Ad($H$)-invariant inner products on $\mm$.
	
	Once we assume that $G < I(M, g)$, again by \cite[Chapter VIII, Lemma 4.2]{doC}, we have that $G/H$ is an \textit{effective presentation}, which means that the ineffective kernel given by $N\coloneqq\{h \in H \ | \ h\cdot q=q, \hspace{0.2cm}\forall q \in M\}=\bigcap_{g \in G} gHg^{-1}$ is trivial. In other words, this means that $G$ and $H$ have no nontrivial common normal subgroup.
	
	On the other hand, by the correspondence above, given a reductive decomposition $\gg=\hh\oplus\mm$ such that $\gg$ and $\hh$ have no nontrivial common ideal, and an inner product $\langle \cdot, \cdot \rangle$ on $\mm$ such that $\Ad(H)$ is a closed subset of $\SO(\mm)$, we can reconstruct the homogeneous Riemannian manifold $(M,g)$ with the \textit{almost-effective presentation} $M=\tilde{G}/\tilde{H}$. Namely, the ineffective kernel $\bigcap_{g \in \tilde{G}} g\tilde{H}g^{-1}$ is discrete, where $\tilde{G}$ and  $\tilde{H}<\tilde{G}$ are integral groups for the Lie algebras $\gg$ and $\hh$, respectively. In this manner, we can restrict the problem to the Lie algebra level. Observe, however, that in this case $\tilde{H}$ is not necessarily compact; for example, it could be the universal cover of a torus. It is clear that we can always assume that a presentation of a homogeneous manifold $M$ is effective or almost-effective, since we can quotient the transitive group $G$ by its ineffective kernel $N$, obtaining an effective presentation given by the still transitive action of $G/N$ on $M$, and analogously for the almost-effectiveness at the Lie algebra level.
	\begin{definition}[$G$-homogeneous manifold]
		We say that a pointed homogeneous manifold $(M,p)$ is a $G$\textit{-homogeneous manifold} if $G$ acts almost-effectively on $M$ and $M$ admits a $G$-invariant Riemannian metric. In this case, we denote $H$ the isotropy group of the $G$ action at $p$ and write also $M= G/H$.
	\end{definition}
\begin{remark*}[Reduction to Lie algebra level] \label{rmk:reduction to lie algebra}
		 By the discussion above, we identify from here onward $G/H$ by a pair consisting of a Lie algebra $\gg$ and a Lie subalgebra $\hh \subset \gg$, such that $\gg$ and $\hh$ have no common ideal, and such that $\Ad(H) \subset \Aut(\gg)$ is compact. Moreover, any $G$-invariant metric on $G/H$ corresponds to an $\Ad(H)$-invariant inner product on a reductive complement $\mm$.
\end{remark*}
	
	We proceed to discuss the Ricci flow on homogeneous manifolds.
	
	Let $(M,g)$ be a homogeneous Riemannian  manifold with an almost-effective presentation $M=G/H$ and reductive decomposition $\gg = \hh \oplus \mm$ on a base point $p$, with the identification $\mm \cong T_pM$. The formula for the Ricci tensor of $(M,g)$ at $X\in\mm$, \cite[Corollary 7.38]{be}, is given by 
	\begin{align}\label{eq:ricci formula}
		\ric_g(X,X)=& - \frac{1}{2}\Bk(X,X) -\frac{1}{2}\sum_i \Vert[X,X_i]_{\mm}\Vert_g^2+ \frac{1}{4}\sum_{i,j}g([X_i,X_j]_{\mm},X)^2 \\ \notag 
		&- g([\Hm_g,X]_{\mm},X).
	\end{align}
	Here $[\cdot,\cdot]_\mm$ is the projection of the Lie brackets according to the reductive decomposition $\hh\oplus\mm$, $\Bk$ is the Killing form of $\gg$, $\{X_i\}_{i=1}^n$ is an orthonormal basis of $\mm$, and $\Hm_g$ is the mean curvature vector defined by $g(\Hm_g,X) \coloneqq \tr(\ad X)$. 
	
	In full generality, the Ricci flow \eqref{eq:general RF} is a nonlinear partial differential equation. In the case where $M$ is compact, Hamilton proved in \cite{ham} the short time existence and uniqueness of it. Later, Shi showed in \cite{shi} that if $(M,g_0)$ is complete, noncompact, with bounded curvature, then the Ricci flow has a solution with bounded curvature on a short time interval, and Chen and Zhu proved in \cite{chenzu} the uniqueness of the flow within this class of complete and bounded curvature Riemannian metrics.
	
	Since every homogeneous Riemannian manifold is complete and has bounded curvature, there exists a unique Ricci flow solution $g(t)$ that is complete with bounded curvature and has an initial homogeneous metric $g_0$. By uniqueness and the diffeomorphism equivariance of the Ricci tensor, it follows that the Ricci flow preserves isometries \cite[Corollary 1.2]{chenzu}. Consequently, this complete bounded curvature Ricci flow solution $g(t)$, with initial $G$-invariant metric $g_0$, remains $G$-invariant. In this setting, we call the Ricci flow \textit{homogeneous}, and the Ricci flow equation reduces to the following autonomous nonlinear ordinary differential equation on the space of $\Ad(H)$-invariant inner products on $\mm$:
	\begin{align}\label{eq:ricci flow}
		\frac{dg(t)}{dt} &= -2\ric(g(t)), \hspace{0.5cm} g(0)=g_0.
	\end{align}
	The Ricci $(0,2)$-tensor in this case can be seen as the following smooth map 
	\[\ric: (\Sym^2(\mm))^H_+ \to (\Sym^2(\mm))^H,\]
	where $(\Sym^2(\mm))^H$ is the nontrivial vector space of $\Ad(H)$-invariant symmetric bilinear forms in $\mm$ and $(\Sym^2(\mm))^H_+$ is the open set of positive definite ones.
	
	Note that by classical ODE theory, given an initial $G$-invariant metric $g_0$ corresponding to an initial $\Ad(H)$-invariant inner product, there is a unique $\Ad(H)$-invariant inner product solution corresponding to a unique family of $G$-invariant metrics $g(t)$ in $M$. And indeed, one can use that to define the homogeneous Ricci flow on the locally homogeneous incomplete case (see \cite{bl18}). 
	
	\subsection{The unimodular Ricci flow}
	
	We can rewrite formula \eqref{eq:ricci formula} for the $(0,2)$-type Ricci tensor of $g$ in its $(1,1)$-type tensor as follows
		\begin{equation}\label{eq:Ricci (1,1) tensor}
		\Ric_g = -\frac{1}{2}\Bk_g+\M_g -S^g(\ad{\Hm_g}),
	\end{equation}
where $\Bk_g$ corresponds to the Killing form, $\M_g $ is the symmetric operator defined as 
\[g\left(\M_gX,X\right) \coloneqq -\frac{1}{2}\sum_{i,j}g(\left[X,X_i\right]_\mm,X_j)^2+\frac{1}{4}\sum_{i,j}g(\left[X_i,X_j\right]_\mm,X)^2,\]
where $\{X_1,\ldots, X_n\}$ is any $g$-orthonormal basis of $\mm$.
And finally, 
\[S^g (\ad{\Hm_g}) \coloneqq \frac{1}{2}\left(\ad\Hm_g +\left(\ad\Hm_g\right)^{t_g}\right)\]
is the symmetrization of $\ad\Hm_g \coloneqq [\Hm_g,\cdot]$, where we denote $(\cdot)^{t_g}$ as the $g$-transpose of a linear map on $\mm$.

\begin{remark*}\label{rem:derivations do nothing}
	Note that the last term involving the mean curvature vector is nothing more than the symmetrization of a derivation of $\gg$ that preserves $\hh$ (see \cite[Corollary 2.7]{bl18}). This means that its infinitesimal action on $g$ can be integrated to a pullback by a one-parameter family of $G$-equivariant diffeomorphisms on $G/H$ (see \cite{jab13}). Therefore, this term does not contribute to any genuine geometric change along the Ricci flow.
\end{remark*}

Let us consider the unimodular part of the Ricci curvature, defined as
\begin{equation}
	\Ric^\star_g \coloneqq \Ric_g + S^g\left(\ad \Hm_g\right)=  -\frac{1}{2}\Bk_g+\M_g,
\end{equation}
which was introduced by Heber in \cite{heb}, and its corresponding $(0,2)$-type tensor 
\begin{equation}
	\ric^\star_g \left(\cdot,\cdot\right) \coloneqq g\left(\Ric_g^\star \cdot,\cdot\right).
\end{equation}

Observe that if the transitive Lie group $G$ is unimodular, then $\Ric^\star = \Ric$. 
\begin{remark*}
	The unimodular Ricci curvature has also been called the \textit{modified} Ricci curvature in \cite{bl18}.
\end{remark*}

With this in hand we have the following natural definition of a homogeneous geometric flow. 
\begin{definition}[Unimodular Ricci flow]
	The following ordinary differential equation
	\begin{equation}\label{eq:unimodular RF}
		\frac{dg}{dt}=-2\ric^\star_g, \hspace{0.5cm} g(0)=g_0
	\end{equation}
	is called the \textit{unimodular Ricci flow}.
\end{definition}

Remark \ref{rem:derivations do nothing} can be seen using the correspondence between the Ricci flow and the \textit{bracket flow} (see \cite{lau13} and \cite{bl18}). Indeed, by \cite[Proposition 2.3]{bl18}, there is a one-to-one correspondence between the space of $G$-invariant metrics on $G/H$ modulo the action of the group of $\Ad(H)$-equivariant automorphisms of $\gg$, $\Aut^H_\mm(\gg)$, and the space of Lie brackets modulo the action of the orthogonal group $\Or(\mm)$.

Now using that this equivalence has a dynamical version, namely, the homogeneous Ricci flow is equivalent to the bracket flow \cite[Theorem 2.5]{bl18}. The change of gauge from the noncompact group $\Aut^H_\mm(\gg)$ for the compact one $\Or(\mm)$, implies that the geometries of the solutions to the ordinary and unimodular Ricci flows with same initial condition are uniformly close to one another \cite[Proposition 3.1]{bl18}, thus yielding the following.
\begin{corollary}\label{cor:uni RF is equivalent to RF}
	Let $M=G/H$ be a $G$-homogeneous manifold with reductive decomposition $\gg=\hh\oplus\mm$. Then the unimodular Ricci flow of $G$-invariant metrics on $M$ is equivalent to the Ricci flow with same initial condition.
\end{corollary}
	
	\section{The unimodular Ricci flow along $\theta$-adapted metrics} \label{section:main intro section}

	In this section, we investigate the long-time behavior of a Ricci flow solution along a special family of homogeneous Riemannian metrics on homogeneous spaces that stem from semisimple representations of compact lie algebras. To this end, we use the equivalence established in the previous section between the Ricci flow and the unimodular Ricci flow, and we study the long-time behavior of the latter to draw conclusions about the former. As we will see, in our setting it is more convenient to obtain dynamical estimates from a solution $g(t)$ of the unimodular Ricci flow, which then allow us to describe the qualitative behavior of the Ricci flow with the same initial condition.
	
	From now on, we consider the case where $M$ is a $G$-homogeneous manifold for a Lie group $G$ with Lie algebra given by the semidirect product $\gg= \uu \ltimes_\theta V$, where $\theta \colon \uu \to \gl\left(V\right)$ is the Lie algebra representation defining it. 
	\begin{definition}[Stable representation]\label{def:stable rep}
		Let $\uu$ be a compact Lie algebra and $V$ a finite dimensional vector space. We say that a representation $\theta \colon \uu \to \gl\left(V\right)$ is \textit{stable} if $\theta(\uu)$ consists of semisimple operators.
	\end{definition}

	Recall that a linear operator $T \colon V \to V$ is semisimple if every $T$-invariant subspace has a complementary $T$-invariant subspace. This is equivalent to the existence of an inner product $\langle \cdot, \cdot \rangle$ on $V$ such that $T$ is normal, i.e., $T$ commutes with its transpose. 
\begin{remark*}\label{remark:root decomp}
	Since we assumed $\uu$ to be compact, there exists an inner product $\langle \cdot, \cdot \rangle$ with respect to which $\theta(\uu)$ consists of normal operators. We have thus a weight space decomposition of $V=\bigoplus_{\a} V^{\a}$ into $\uu$-submodules $V^\a$, for linear functionals $\a \in \uu^*$, such that for all $X \in \uu$, 
	\[\left.\ad X\right|_{V^\a} = \a(X)\Id+\J^\a_X, \]
	where $\J^\a_X$ is a semisimple operator with purely imaginary eigenvalues.
\end{remark*}
\begin{remark*} \label{rem:stable = there is minimum}
	The name in Definition \ref{def:stable rep} comes from the following fact. The Lie algebra $\uu$ acts  on $V$ via normal operators with respect to a given inner product $\langle \cdot,\cdot\rangle$ if and only if the associated representation $\theta \colon \uu \to \gl(V)$ is a minimal element of the orbit $\GL(V)\cdot \theta$, where $\GL(V)$ acts on $\gl(V)$ via conjugation. This minimality is taken with respect to a canonical metric on $\End\left(\uu,\gl(V)\right)$, induced by the trace metric determined by $\langle \cdot,\cdot\rangle$ on $\gl(V)$ and a background metric on $\uu$ (see \cite{jp}). By geometric invariant theory, this is equivalent to the orbit $\GL(V)\cdot \theta$ being closed (see \cite[Theorem 4.3]{rs} and also \cite[Theorem 1.1.1]{bl20}). 
\end{remark*}

Let us assume then that $G=U\ltimes_\theta V$, where $U$ is an integral subgroup of the compact Lie algebra $\uu$, where we are identifying $V\cong\exp(V)$, and where we abuse notation to also denote by $\theta$ the group-level representation $\theta \colon U \to \GL\left(V\right)$.

\begin{definition}[Stable $U\ltimes_\theta V$-homogeneous manifold]\label{def:GIT-stable}
	We say a $U\ltimes_\theta V$-homogeneous manifold is \textit{stable} if $\theta$ 
	is stable.
\end{definition}

Observe that we have a decomposition $\uu = \kk \oplus \hat{\zz}$, where $\kk$ is compactly embedded in $\uu\ltimes_\theta V$ containing $\uu_{ss}\coloneqq\left[\uu,\uu\right]$, and $\hat{\zz}$ is the subalgebra of the center $\zz(\uu)$ of $\uu$ consisting of elements $Y \in \hat{\zz} \subset \zz(\uu)$ such that $\a(Y)\neq0$ for some $\a \in \uu^*$. It is then clear that $\theta$ is stable if and only if $\theta$ restricted to $\hat{\zz}$ is. 

Let then $M=G/H$ be a stable $U\ltimes_\theta V$-homogeneous manifold and let us also consider the Lie groups $K$ and $\hat{Z}$ with respective Lie algebras $\kk$ and $\hat{\zz}$. By almost-effectiveness, we can assume without loss of generality that $H\subset K$ and that $V$ is the abelian nilradical of $\gg=\uu\ltimes_\theta V$. 
	\begin{remark*}\label{rem:homogeneous bundle}
			Observe that $M$ is, up to coverings, a $U$-homogeneous principal $\mathbb{R}^{\dim V}$-bundle over the homogeneous space $U/H = K/H\times \hat{Z}$. Also, on another perspective, $M$ is a $K$-homogeneous principal bundle over the compact homogeneous base $K/H$ with the solvmanifold $\hat{Z}\ltimes V$ as fiber.
	\end{remark*}
	
	\begin{notation}\label{not:reduc decomp for polar}
		Let $M$ be a stable $U\ltimes_\theta V$-homogeneous manifold. We consider from here onward the following reductive decomposition for $M=G/H$, where $G=U\ltimes V$,
		\begin{equation*}\label{eq:polar red decomp}
			\gg = \hh \oplus \ll \oplus \hat{\zz} \oplus V, \hspace{1cm} \kk = \hh\oplus\ll,
		\end{equation*}
		where $\ll$ is an $\Ad(H)$-submodule complementary  to $\hh$ in $\kk$. We moreover denote 
		\begin{equation*}\label{eq:m_u}
			\mm_\uu\coloneqq\ll \oplus \hat{\zz}
		\end{equation*}
		as our choice of reductive complement to $\hh$ in $\uu$, and pick 
		\begin{equation*}\label{eq:mm=ll+zz+V}
			\mm \coloneqq \ll  \oplus \hat{\zz} \oplus V
		\end{equation*}
		as our reductive complement to $\hh$ in $\gg$.
	\end{notation}
	
	We restrict our attention to $U\ltimes_{\theta} V$-invariant metrics whose geometry is \textit{compatible} with the stable representation $\theta$. That is, we consider $U\ltimes V$-invariant metrics $g$ for which there is a weight space decomposition $\bigoplus_{r=1}^p V^{\alpha_r}$, as in Remark \ref{remark:root decomp}, such that $V^{\a_1} \perp_g \ldots \perp _gV^{\a_p}$. Moreover, we require that the metric $g$ is well adapted to the decomposition $\uu\ltimes V$, namely that $\mm_\uu \perp_g V $. 
	 
	 We want first to show that these metrics are (unimodular) Ricci flow invariant. For the case where $\hat{\zz}\neq0$,  Will and E. Lauret, in \cite{will} and in \cite{lw}, discovered in this family of metrics examples of Ricci negative left-invariant metrics on nonsolvable, nonsemisimple Lie groups (see \cite[Theorem 3.3]{will} and \cite[Theorem 5.1]{lw}). The smallest dimensional example of those constructions is $\left(\SU(2)\times \mathbb{R}\right)\ltimes \mathbb{R}^3$, where $\mathbb{R}^3$ is an irreducible $\SU(2)$-representation and $\mathbb{R}$ acts on $\mathbb{R}^3$ via a multiple of the identity. If $\hat{\zz}=0$, then $\gg$ is unimodular. So, by \cite[Theorem 2]{dot}, $M=G/H$ does not admit Ricci negative metrics, as there exists a Ricci nonnegative direction in $V$. Thus, for example, in relation to the recently solved Alekseevskii conjecture \cite[Theorem A]{bl23}, these spaces pose no difficulty. On the other hand, confirming the dynamical Alekseevskii conjecture on them, even in the special case that the metric $g$ satisfies $\ll \perp_g V^{\a_1} \perp_g \ldots \perp_g V^{\a_p} $, is not trivial. 
	\begin{notation}\label{not:orth frame}
		Let $M$ be a stable $U\ltimes_\theta V$-homogeneous manifold with a weight space decomposition $\bigoplus_{\a \in \mathcal{I}^*} V^{\a}$, where $\mathcal{I}^* \coloneqq \{\a_1, \ldots,\a_p\} \subset \uu^*$ and $\theta$ is as in Remark \ref{remark:root decomp}. Fix once and for all an $\Ad(K)$-invariant background metric $\langle \cdot,\cdot\rangle$ on $\gg$ such that $\kk \perp \hat{\zz} \perp V^{\alpha_1} \perp \ldots \perp V^{\alpha_p}$, and let $\ll \coloneqq \hh^\perp \subset \kk$. Then, with respect to this background metric, $\hh \perp \ll\perp\hat{\zz} \perp V^{\alpha_1}\perp \ldots \perp V^{\alpha_p} $. Recall that we denote $\mm \coloneqq \ll  \oplus \hat{\zz} \oplus V$ and $\mm_\uu\coloneqq\ll\oplus\hat{\zz}$.
		
		We then establish the following notation. Let $g(\cdot,\cdot) = \langle P\cdot, \cdot\rangle$, where $P$ is an $\Ad(H)$-equivariant, positive-definite operator on $\mm$ such that $\mm_\uu \perp_g V^{\a_1}\perp_g \ldots \perp_g V^{\a_p}$. Since $\ll \perp_g V$, let $\{\bar{L}_i\}_{i=1}^{m}$, where $m = \dim \ll$, be a diagonalizing $\langle \cdot,\cdot\rangle$-orthonormal frame for $P^{\ll}\coloneqq\pr_{\ll}\circ \left.P\right| _{\ll}$, where $\pr_{\ll}$ is the $\langle\cdot,\cdot\rangle$-orthogonal projection to $\ll$, with respective eigenvalues $0<g^\ll_1\leq \ldots \leq g^\ll_m$. 
		Let, analogously, $\{\bar{A}^{\a}_i\}_{i=1}^{d_\a}$, where $d_\a= \dim V^\a$, be a diagonalizing $\langle \cdot,\cdot\rangle$-orthonormal frame for $P^{V^\a}\coloneqq\pr_{V^\a}\circ \left.P\right| _{V^\a}$
		with respective eigenvalues $0<g_1^\a \leq \ldots \leq g_{d_\a}^\a $. If we define $L_i :=\frac{\bar{L}_i}{\sqrt{g^\ll_i}}$, for $1 \leq i \leq m$, and complete it with a $g$-orthonormal frame for the orthogonal complement $\ll^\perp \subset \ll\oplus\hat{\zz}=\mm_\uu$, we get a $g$-orthonormal frame $\{U_i\}_{i=1}^{\dim\mm_\uu}$ for $\mm_\uu$. Note that $\ll$ is not necessarily $g$-perpendicular to $\hat{\zz}$, even though this is the case for the background metric $\langle \cdot,\cdot\rangle$. We moreover define $A^\a_i :=\frac{\bar{A}^\a_i}{\sqrt{g^\a_i}}$, for $1 \leq i \leq d_\a$, to get a $g$-orthonormal frame for each $V^\a$, $\a \in \mathcal{I}^* = \{\a_1,\ldots,\a_p\}$.
		In this way, we obtain a $g$-orthonormal frame $ \{X_i\}_{i=1}^{\dim \mm}\coloneqq\{U_i\}\cup\{A_i\}$ for $\mm$, adapted to the geometric condition $\mm_\uu\perp_g V^{\a_1} \perp_g \ldots \perp_g V^{\a_p} $, which will be used repeatedly below to express the Ricci curvature of $M$.
	\end{notation}
	
	Let us now recall the Ricci tensor formula \eqref{eq:ricci formula}
	\[\ric_g(X,X)= - \frac{1}{2}\Bk(X,X) -\frac{1}{2}\sum_i \Vert[X,X_i]_{\mm}\Vert_g^2+ \frac{1}{4}\sum_{i,j}g([X_i,X_j]_{\mm},X)^2 - g([\Hm_g,X]_{\mm},X),\]
	which we rewrite for convenience in the polarized form, for all $X,Y \in \mm$, as
	\begin{equation}\label{eq:ricci XY formula}
		\ric_g(X,Y) = -\frac{1}{2}\Bk(X,Y)+\mo_g(X,Y)-\h_g(X,Y),
	\end{equation}
	where
	\begin{equation}
		\mo_g(X,Y) \coloneqq -\frac{1}{2}\sum_i g\left(\left[X,X_i\right]_\mm,\left[Y,X_i\right]_\mm\right)+\frac{1}{4}\sum_{i,j} g\left(\left[X_i,X_j\right]_\mm,X\right)g\left(\left[X_i,X_j\right]_\mm,Y\right)
	\end{equation}
	and
	\begin{equation}
		\h_g(X,Y) \coloneqq g\left(S^g\left(\ad{\Hm_g}\right)(X),Y\right)= \frac{1}{2}\left(g\left(\left[\Hm_g,X\right]_\mm,Y\right)+g\left(\left[\Hm_g,Y\right]_\mm,X\right)\right).
	\end{equation}
	With this notation we have that 
	\begin{equation}\label{eq:uni Ricci formula}
		\ric^\star_g (X,Y)=-\frac{1}{2}\Bk(X,Y) +\mo_g(X,Y).
	\end{equation}
	
	\begin{proposition} \label{prop:Ricci invariant semidirect}
		Let $M$ be a stable $U\ltimes_\theta V$-homogeneous manifold with a weight space decomposition $V=\bigoplus_{\a \in \mathcal{I}^*} V^{\a}$, $\mathcal{I}^* = \{\a_1,\ldots,\a_p\}$,  and reductive decomposition $\gg=\hh\oplus\mm_\uu\oplus V$. Then the space of $U\ltimes V$-invariant metrics on $M$ such that $\mm_\uu \perp V^{\a_1} \ldots \perp V^{\a_p} $ is Ricci flow invariant.
	\end{proposition}
	\begin{proof}
		Since $V \subset \gg$ is an abelian ideal, $\tr(\ad A)=0$ for all $A \in V$. Thus, for any $U\ltimes V$-invariant metric $g$ such that $\mm_\uu \perp_g V^{\a_1} \perp_g  \ldots \perp_g V^{\a_p}$, 
		\begin{equation*}
			g\left(\Hm_g,A\right) \coloneqq \tr (\ad A)=0,
		\end{equation*}
		i.e., the mean curvature vector $\Hm_g \in \mm_\uu$.
		
		Let $Y \in \mm_\uu$ and $A \in V$, then the Ricci tensor formula \eqref{eq:ricci XY formula} for the orthonormal frame $\{X_i\}$ as in Notation \ref{not:orth frame} gives us
		\begin{align*}
			\ric_g(Y,A)=&  - \frac{\Bk(Y,A)}{2}-\frac{1}{2}\sum_i g([Y,X_i]_{\mm}, [A,X_i]_{\mm}) +\frac{1}{4}\sum_{i,j}g([X_i,X_j]_{\mm},Y)g([X_i,X_j]_{\mm},A)  \\
			&- \frac{1}{2}\left(g\left(\left[\Hm_g,Y\right]_{\mm}, A\right)+g\left(\left[\Hm_g,A\right]_{\mm}, Y\right)\right)\\
			=& - \frac{\Bk(Y,A)}{2}-\frac{1}{2}\sum_i g([Y,X_i]_{\mm}, [A,X_i]_{\mm})+\frac{1}{4}\sum_{i,j}g([X_i,X_j]_{\mm},Y)g([X_i,X_j]_{\mm},A),
		\end{align*}
		since $\Hm_g \in \mm_\uu \subset \uu$,  $\left[\uu,\uu\right] \subset \uu$ and $\left[\gg,V\right]\subset V$.
		
		Note that for an abelian ideal $\mathfrak{a}$ we have that $\Bk (\mathfrak{a}, \gg)=0$. This can be seen since for any $X \in \gg$ and $A \in \aa$, $\left(\ad X\circ \ad A\right)(\gg) \subset \aa$, thus $\left(\ad X\circ \ad A\right)^2=0$, i.e., $\ad X\circ \ad A$ is nilpotent, hence traceless. Thus, using that $V$ is an abelian ideal of $\gg$ and the $g$-orthonormal basis $\{X_i\}$ above, we get that
		\begin{align*}
			\ric_g(Y,A)=&  - \frac{\Bk(Y,A)}{2}-\frac{1}{2}\sum_i g([Y,X_i]_{\mm}, [A,X_i]_{\mm}) +\frac{1}{4}\sum_{i,j}g([X_i,X_j]_{\mm},Y)g([X_i,X_j]_{\mm},A)  \\
			=&-\frac{1}{2}\sum_i g([Y,X_i]_{\mm}, [A,X_i]_{\mm})+\frac{1}{4}\sum_{i,j}g([X_i,X_j]_{\mm},Y)g([X_i,X_j]_{\mm},A)\\
			=& -\frac{1}{2}\left(\sum_i g([Y,U_i]_{\mm}, [A,U_i]_{\mm})+\sum_i g([Y,A_i]_{\mm}, [A,A_i]_{\mm})\right)\\
			&+\frac{1}{4}\sum_{i,j}g([X_i,X_j]_{\mm},Y)g([X_i,X_j]_{\mm},A) \\
			=& \ \frac{1}{2}\sum_{i,j,\a}g([A^\a_i,X_j]_{\mm},Y)g([A^\a_i,X_j]_{\mm},A)+\frac{1}{4}\sum_{i,j}g([U_i,U_j]_{\mm},Y)g([U_i,U_j]_{\mm},A)\\
			=&\ 0,
		\end{align*}
	where in the last two equalities we used that $[\mm_\uu,\mm_\uu]_{\mm} \subset \mm_\uu$ and that $\mm_\uu \perp_g V$.
		
		Now let  $A^\a \in V^\a$, $A^\b \in V^\b$, $\a \neq \b$. Since  $\Bk\left(V, \cdot\right)=0$ and each $V^\gamma$ is a $\gg$-submodule, we have that
		\begin{align*}
			\ric_g(A^\a,A^\b)=&  - \frac{\Bk(A^\a,A^\b)}{2}-\frac{1}{2}\sum_i g([A^\a,X_i]_{\mm}, [A^\b,X_i]_{\mm})\\ &+\frac{1}{4}\sum_{i,j}g([X_i,X_j]_{\mm},A^\a)g([X_i,X_j]_{\mm},A^\b)  \\
			&- \frac{1}{2}\left(g\left(\left[H_g,A^\a\right]_{\mm}, A^\b\right)+g\left(\left[H_g,A^\b\right]_{\mm}, A^\a\right)\right)\\
			=&\ \frac{1}{2}\sum_{i,j,\gamma}g\left(\left[U_i,A^\gamma_j\right]_{\mm},A^\a\right)g\left(\left[U_i,A^\gamma_j\right]_{\mm},A^\b\right)=0.
		\end{align*}
		
		This means that for all $\a,\b \in \mathcal{I}^*$,
		\[\frac{dg(0)}{dt}(V^{\a},V^{\b})=0=\frac{dg(0)}{dt}(\mm_\uu,V).\]
		Thus, from classic ODE theory, the set of metrics $\{g \in \text{Sym}^2(\mm)^{H}_+ \ \vert \ \mm_\uu\perp_g V^{\a_1} \perp_g \ldots \perp_g V^{\a_p}\}$ is an invariant subset for the Ricci flow equation \eqref{eq:ricci flow}. \\
	\end{proof}
	
	\begin{corollary}\label{cor: polar is uni RF invariant}
		Let $M$ be a stable $U\ltimes_\theta V$-homogeneous manifold with a weight space decomposition $V=\bigoplus_{\a \in \mathcal{I}^*} V^{\a}$, $\mathcal{I}^* = \{\a_1,\ldots,\a_p\}$,  and reductive decomposition $\gg=\hh\oplus\mm_\uu\oplus V$. Then the space of $U\ltimes V$-invariant metrics on $M$ such that $\mm_\uu \perp V^{\a_1} \ldots \perp V^{\a_p} $ is unimodular Ricci flow invariant.
	\end{corollary}
	\begin{proof}
		First of all, recall that the unimodular Ricci tensor is $\Ad(H)$-invariant \cite[Lemma 2.6]{bl18}. This implies, by uniqueness of the solution, that the unimodular Ricci flow indeed preserves $\Ad(H)$-invariant inner products on $\mm$.
		
		As we have seen in the computation of Proposition \ref{prop:Ricci invariant semidirect}, for all $Y \in \mm_\uu$ and $A\in V$
		\[\Bk(Y,A) = \mo_g(Y,A)=0.\]
		As well as 
		\[\Bk(A^\alpha,A^\beta)= \mo_g(A^\alpha,A^\beta)=0,\]
		for all $A^\alpha \in V^\alpha$ and $A^\beta \in V^\beta$, $\a\neq\b$.
		
		Thus the same is true for the unimodular Ricci tensor $\ric^\star_g = -\frac{1}{2}\Bk + \mo_g$, which implies that the set of metrics $\{g \in \text{Sym}^2(\mm)^{H}_+ \ \vert \ \mm_\uu\perp_g V^{\a_1} \perp_g \ldots \perp_g V^{\a_p}\}$ is an invariant subset for the unimodular Ricci flow equation \eqref{eq:unimodular RF}.
	\end{proof}
	
	\begin{definition}[$\theta$-adapted standard metrics]\label{def:polar metrics}
		Let $M$ be a stable $U\ltimes_\theta V$-homogeneous manifold.
		We call \textit{$\theta$-adapted standard metrics} the set of $U\ltimes V$-invariant metrics $\mathfrak{S}_{U\ltimes V}$ such that, for each $g \in \mathfrak{S}_{U\ltimes V}$, there is a weight space decomposition $V=\bigoplus_{\a \in \mathcal{I}^*} V^{\a}$, $\mathcal{I}^* = \{\a_1,\ldots,\a_p\}$, and a reductive decomposition $\gg=\hh\oplus\mm_\uu\oplus V$ such that $\mm_\uu\perp_g V^{\a_1} \perp_g \ldots \perp_g V^{\a_p}$.
	\end{definition}

Note that $\theta$-adapted standard metrics are, in particular, \textit{standard} in the sense of \cite{heb}; namely, the isometric left action of the abelian nilradical $V$ has an \textit{integrable horizontal distribution} (see \cite[Definition 2.1]{bl22}). Moreover, they are analogous to \textit{awesome metrics} on semisimple homogeneous spaces---in which the Cartan decomposition is an orthogonal decomposition---in the sense that they require a compatibility condition between the metric and the representation of $\kk$ on $\gg$. Finally, as in the case of awesome metrics, they have a nice splitting of the Ricci tensor and therefore provide good initial conditions for the analysis of the long-time behavior of the Ricci flow (see e.g. \cite{n} and \cite{ara2} for the awesome case).
	\begin{definition}[$g$-adapted reductive complement]\label{def:adapted complement}
		Let $g$ be a $\theta$-adapted standard metric. We say that a reductive decomposition $\gg=\hh\oplus\mm_\uu \bigoplus_{\a \in \mathcal{I}^*} V^{\a}$ composed of a weight space decomposition $V=\bigoplus_{\a \in \mathcal{I}^*} V^{\a}$, $\mathcal{I} = \{\a_1,\ldots,\a_p\}$, as in Remark \ref{remark:root decomp}, and a reductive $\hh$-complement $\mm_{\uu} \subset \uu$ is \textit{adapted} to $g$ when $\mm_\uu\perp_g V^{\a_1} \perp_g \ldots \perp_g V^{\a_p}$.
	\end{definition}
	
	\section{Algebraic bounds for the unimodular Ricci curvature}\label{section:algebraic bounds}
	In order to prove Theorem \ref{thm:A}, we need first to get some a priori algebraic estimates for the Ricci curvature of these metrics both in the abelian noncompact part corresponding to $V$ and the compact part corresponding to $K/H$. 
	
	\begin{proposition}\label{prop:A bounds}
		Let $(M,g)$ be a stable $U\ltimes_\theta V$-homogeneous manifold with a $\theta$-adapted standard metric $g$ and let $\gg=\hh\oplus\mm_\uu \bigoplus_{\a \in \mathcal{I}^*} V^{\a}$ be a reductive decomposition adapted to $g$. Let $\{\bar{A}_i^\alpha\}_{i=1}^{d_\a}$ be a $\left.g\right\vert_{V^\a\times V^\a}$-diagonalizing $\langle\cdot,\cdot\rangle$-orthonormal frame for $V^\a$, with respective eigenvalues $g_i^\alpha$, and let $A_i^\alpha \coloneqq \frac{\bar{A}_i^\alpha}{\sqrt{g_i^\alpha}}$ (see Notation \ref{not:orth frame}). Let moreover $\{U_k\}_{k=1}^{\dim \mm_\uu}$ be a $g$-orthonormal basis for $V^\perp=\mm_\uu$.
		
		Then,
		\begin{equation}\label{eq:Ricci-A formula}
			\ric^\star\left(A^\alpha_i,A^\alpha_i\right) = \frac{1}{2}\sum_{\substack{j, k}} \left\langle \left[U_k, \bar{A}^\a_i\right],\bar{A}^\alpha_j \right\rangle^2\left(\frac{g^\alpha_i}{g^\a_j}-\frac{g^\a_j}{g^\alpha_i}\right).
		\end{equation}
	\end{proposition}
	
	\begin{proof}
		Recall first that the abelian ideal $V$ is contained in the kernel of the Killing form $\Bk$. Moreover, since $V=V^{\a_1} \perp_g \ldots \perp_g V^{\a_p}$, a direct computation using the unimodular Ricci tensor formula \eqref{eq:uni Ricci formula} with our special basis yields the following expression in the directions $A_i^\alpha$,
		\begin{align*}
			\ric^\star\left(A^\alpha_i,A^\alpha_i\right)&= -\frac{1}{2}\sum _{\substack{\beta, j, k}}g\left(\left[A^\alpha_i,U_k\right],A^\beta_j\right)^2 +\frac{1}{2}\sum_{\substack{\beta, j, k}}g\left(A^\alpha_i,\left[U_k,A^\beta_j\right]\right)^2 +\frac{1}{4}\sum_{j,k}g\left(A^\alpha_i,\left[U_j,U_k\right]_\mm\right)^2 \\
			&= -\frac{1}{2}\sum _{\substack{j, k}}g\left(\left[A^\alpha_i,U_k\right],A^\a_j\right)^2 +\frac{1}{2}\sum_{\substack{j, k}}g\left(A^\alpha_i,\left[U_k,A^\a_j\right]\right)^2.
		\end{align*}
		
		By hypothesis, $U_k$ acts on $V^\a$ as $\left.\ad U_k\right|_{V^\a}=\a(U_k)\Id + \J^\a_{U_k}$, where $\J^\a_{U_k}$ is a skew-symmetric operator with respect to the background metric $\langle \cdot,\cdot \rangle$.
		
		Thus, for $i \neq j$, 
		\[\left\langle \left[U_k, A_i^\alpha\right], A_j^\a \right \rangle^2 = \left\langle \J^\a_{U_k} A_i^\alpha, A_j^\a \right\rangle^2=\left \langle A_i^\alpha, \J^\a_{U_k} A_j^\a \right \rangle^2=\left \langle A_i^\alpha, \left[U_k, A_j^\a \right] \right \rangle^2,\]
		since $\langle \cdot, \cdot \rangle$ is $\ad\left(\kk\right)$-invariant. 
		
		Therefore, using our $g$-diagonalizing basis $\bar{A}^\a_i$ with respective eigenvalue $g_i^\a$, we have that 
		\begin{align*}
			\ric^\star\left(A^\alpha_i,A^\alpha_i\right)&= \frac{1}{2}\sum _{j,k}\left(g\left(A_i^\alpha,\left[U_k,A^\a_j\right]\right)^2 -g\left(\left[A^\alpha_i,U_k\right],A^\a_j\right)^2\right)\\
			&=\frac{1}{2}\sum _{j,k}\left(g_i^\alpha\left \langle \bar{A}_i^\alpha,\left[U_k,A^\a_j\right]\right\rangle^2 -g_j^\a\left \langle \left[A^\alpha_i,U_k\right],\bar{A}^\a_j\right\rangle^2\right)\\
			&=\frac{1}{2}\sum_{j,k} \left\langle \left[U_k, \bar{A}^\a_j\right],\bar{A}_i^\alpha \right\rangle^2\left(\frac{g^\alpha_i}{g^\a_j}-\frac{g^\a_j}{g^\alpha_i}\right).
		\end{align*}
	\end{proof}
	
	\begin{remark*}\label{remark:minimal is equivalent to}
		It is worth mentioning that if $\theta \colon \uu \to \gl(V)$ is the representation $\theta(U) \coloneqq \left.\ad U\right|_{V}$, then for $A \in V$
		\[\ric^\star(A,A)= \frac{1}{2}g\left(\sum_k\left[\theta\left(U_k\right),\theta\left(U_k\right)^{t_g}\right]A,A\right).\]
		Thus, $\left.\ric^\star\right|_V =0$ precisely when the representation $\theta$ is a minimal point for the action by conjugation of $\GL(V)$ on the space of endomorphisms $\End\left(\uu, \gl(V)\right)$ with the canonical trace inner product induced by the metric $g$. The formula \eqref{eq:Ricci-A formula} says that under our hypothesis on the representation $\theta$, if $\left.g\right|_{V\times V}$ is a metric such that $\theta$ acts as normal operators, then $\theta$ is minimal. Conversely, by \cite[Proposition 3.9]{jp}, if $\theta$ is minimal, since $\uu$ is compact, then $\theta$ acts by normal operators. Furthermore, since the minimality condition in this case does not depend on the metric $\left.g\right\vert_{\mm_\uu\times\mm_\uu}$, this means that for any given stable $\theta$ there is a minimal representation $\theta_{\min}$ conjugated to it. 
	\end{remark*}
	
	A direct application of the formula above entails the following corollary.
	
	\begin{corollary} \label{cor:Ricci-A and Ricci A-bar}
		Let $(M,g)$ be a stable $U\ltimes_\theta V$-homogeneous manifold with a $\theta$-adapted standard metric $g$ and let $\gg=\hh\oplus\mm_\uu \bigoplus_{\a \in \mathcal{I}^*} V^{\a}$ be a reductive decomposition adapted to $g$. 
		Let $\{\bar{A}_i^\alpha\}_{i=1}^{d_\a}$ be a $\left.g\right\vert_{V^\a\times V^\a}$-diagonalizing $\langle\cdot,\cdot\rangle$-orthonormal frame, with respective ordered eigenvalues $g_1^\a \leq \ldots \leq g_{d_\a}^{\a}$, and $A_i^\alpha \coloneqq \frac{\bar{A}_i^\alpha}{\sqrt{g_i^\alpha}}$ (see Notation \ref{not:orth frame}). Then for any $i_0 \leq d_\a$,
		\begin{equation}\label{ineq:Ricci-A is pos}
			\sum_{i=i_0} ^{d_\a} \ric ^\star \left(A_i^\alpha,A_i^\alpha\right)=\frac{1}{2}\sum_{\substack{k}}\sum_{\substack{i\geq i_0\\ j < i_0}}\left\langle \left[U_k, \bar{A}^\alpha_i\right],\bar{A}^\alpha_j \right\rangle^2\left(\frac{g^\alpha_i}{g^\alpha_j}-\frac{g^\alpha_j}{g^\alpha_i}\right)\geq 0.
		\end{equation}
		Moreover,
		\begin{equation}\label{ineq: Ricci-A-bar}
			\sum_{i=i_0} ^{d_\a} \ric ^\star \left(\bar{A}_i^\alpha,\bar{A}_i^\alpha\right)\geq \sum_{\substack{k}}\sum_{\substack{i\geq i_0\\ j < i_0}}\left\langle \left[U_k, \bar{A}^\alpha_i\right],\bar{A}^\alpha_j \right\rangle^2\left(\frac{\left(g_i^{\a}\right)^2}{g_j^\a}-g_j^\a\right)\geq0.
		\end{equation}
	\end{corollary}
	\begin{proof}
		By formula \eqref{eq:Ricci-A formula} we have that
		\begin{align*}
			2\sum_{i=i_0}^{d_\a} \ric^\star \left(A^\alpha_i,A^\alpha_i\right) =& \sum_{\substack{k}}\sum_{\substack{i\geq i_0\\ j < i_0}}\left\langle \left[U_k, \bar{A}^\alpha_i\right],\bar{A}^\alpha_j \right\rangle^2\left(\frac{g^\alpha_i}{g_j^\alpha}-\frac{g_j^\alpha}{g_i^\alpha}\right)\\
			&+ \sum_{\substack{k}}\sum_{\substack{i\geq i_0\\ j \geq i_0}}  \left\langle \left[U_k, \bar{A}^{\a}_i\right],\bar{A}_j^{\a} \right\rangle^2\left(\frac{g_i^\a}{g^\a_{j}}-\frac{g^\a_j}{g^\a_i}\right)\\
			=&\sum_{\substack{k}}\sum_{\substack{i\geq i_0\\ j < i_0}}\left\langle \left[U_k, \bar{A}^\alpha_i\right],\bar{A}^\alpha_j \right\rangle^2\left(\frac{g^\alpha_i}{g_j^\alpha}-\frac{g_j^\alpha}{g_i^\alpha}\right) \geq 0,
		\end{align*}
		where in the second equality we used that the term $\left\langle\left[U_k,A_i^\alpha\right],A_j^\a\right\rangle^2$ is symmetric in the indices $i$ and $j$, as shown in the proof of Proposition \ref{prop:A bounds}.
		
		Analogously,
			\begin{align*}
			2\sum_{i=i_0}^{d_\a} \ric^\star \left(\bar{A}^\alpha_i,\bar{A}^\alpha_i\right) =&\sum_{\substack{k}}\sum_{i\geq i_0}g_i^{\a}\left(\sum_j\left\langle \left[U_k, \bar{A}^\alpha_i\right],\bar{A}^\alpha_j \right\rangle^2\left(\frac{g^\alpha_i}{g_j^\alpha}-\frac{g_j^\alpha}{g_i^\alpha}\right)\right)\\
			=& \sum_{\substack{k}}\sum_{\substack{i\geq i_0\\ j < i_0}}\left\langle \left[U_k, \bar{A}^\alpha_i\right],\bar{A}^\alpha_j \right\rangle^2\left(\frac{\left(g_i^{\a}\right)^2}{g_j^\a}-g_j^\a\right)\\
			&+ \sum_{\substack{k}}\sum_{\substack{i> j\geq i_0}}  \left\langle \left[U_k, \bar{A}^{\a}_i\right],\bar{A}_j^{\a} \right\rangle^2\left(\frac{\left(g_i^{\a}\right)^2}{g_j^\a}-g_j^\a+\frac{(g^\a_j)^2}{g^\a_i}-g^\a_i\right)\\
			=& \sum_{\substack{k}}\sum_{\substack{i\geq i_0\\ j < i_0}}\left\langle \left[U_k, \bar{A}^\alpha_i\right],\bar{A}^\alpha_j \right\rangle^2\left(\frac{\left(g_i^{\a}\right)^2}{g_j^\a}-g_j^\a\right) \\
			&+ \sum_{\substack{k}}\sum_{\substack{i> j\geq i_0}}  \left\langle \left[U_k, \bar{A}^{\a}_i\right],\bar{A}_j^{\a} \right\rangle^2\frac{1}{g_i^{\a}g_j^{\a}}\left((g_i^{\a})^3 - g_i^\a(g_j^{\a})^2+(g_j^{\a})^3-g_j^\a(g_i^{\a})^2\right)\\
			\geq& \sum_{\substack{k}}\sum_{\substack{i\geq i_0\\ j < i_0}}\left\langle \left[U_k, \bar{A}^\alpha_i\right],\bar{A}^\alpha_j \right\rangle^2\left(\frac{\left(g_i^{\a}\right)^2}{g_j^\a}-g_j^\a\right)  \geq 0,
		\end{align*}
	where in the second to last inequality we used that $(a^3-ab^2+b^3-ba^2)=(a-b)^2(a+b)\geq0$, when $a,b \geq0$.

	\end{proof}
	\begin{remark*}\label{remark:Ricci-A bounds for smallest eigen}
		Notice that the nonnegativity of $\ric^\star \left(\bar{A}^\alpha_{d_\alpha},\bar{A}^\alpha_{d_\alpha}\right)$ is not enough to apply an argument analogous to the one in the proof of \cite[Theorem 3.2]{boe}. This is so because this nonnegative term does not come from the Killing form and thus it can go to zero too quickly along the flow, indeed this term can be zero.
		
		Moreover, note that we could do an analogous computation, as the one done to obtain the estimate \eqref{ineq:Ricci-A is pos}, in order to show that the sum of the unimodular Ricci curvatures starting from the lowest eigenvalues is nonpositive. Furthermore, for the smallest eigenvalue $g^\a_1$ with eigenvector $\bar{A}^\a_1$, we have
		\[\ric^*\left(\bar{A}^\a_1,\bar{A}^\a_1\right)\leq0.\]
	\end{remark*}
	\begin{remark*}
		Observe that the computation in the proof of Corollary \ref{cor:Ricci-A and Ricci A-bar} yields 
		\begin{equation*}
			\tr \left(\left. \Ric\right|_{V^\a}\right)=\sum_{i=1} ^{d_\a} \ric ^\star \left(A_i^\alpha,A_i^\alpha\right) =0.
		\end{equation*}
	And it is easy to see that for a general $U\ltimes V$-invariant metric $g$ and a $g$-orthonormal basis $\{A_i\}$ for $V$, the computation above yields
		\[	\tr \left(\left. \Ric\right|_{V}\right)= \frac{1}{4}\sum_{i,j,k}g\left(\left[Y_j,Y_k\right]_\mm,A_i\right)^2\geq 0,\]
		where $\{Y_j\}$ is a $g$-orthonormal basis for $V^{\perp_g}$.
		This was observed by Dotti, Leite, Miatello in \cite{dlm}, where they showed that if a unimodular Lie group admits a left-invariant Ricci negative metric, then it must be semisimple. 
		Observe that $\sum_i\ric^\star(A_i,A_i)=0$ if and only if the horizontal distribution $V^{\perp_g}$ is integrable, which is precisely the case we investigate here.
	\end{remark*}
	\begin{proposition} \label{prop:Ricci-U}
		Let $(M,g)$ be a stable $U\ltimes_\theta V$-homogeneous manifold with a $\theta$-adapted standard metric $g$ and let $\gg=\hh\oplus\mm_\uu \bigoplus_{\a \in \mathcal{I}^*} V^{\a}$ be a reductive decomposition adapted to $g$. 
		Then for $Y \in \mm_\uu$,

		\begin{equation}\label{eq:Ricci-U formula}
			\ric^\star(Y,Y) =\ric_{U/H}(Y,Y)- \frac{1}{2}\left(\sum_{\alpha,i}\Vert\left[Y,A^\alpha_i\right]\Vert_g^2+\tr\left(\left.\ad^2 Y\right\vert_{V}\right)\right),
		\end{equation}
		where  $\{A^\alpha_i\}_{i=1}^{d_\a} \subset V^\a$ is a $\left.g\right\vert_{V^\a\times V^\a}$-orthonormal frame and $\ric_{U/H}$ is the Ricci tensor for the homogeneous submanifold $(U/H, \left.g\right\vert_{U/H}) \subset (G/H, g)$.
	\end{proposition}
	
	\begin{proof}
		Using the unimodular Ricci tensor formula \eqref{eq:uni Ricci formula} with a $g$-diagonalizing basis as in Notation \ref{not:orth frame}, and denoting the Killing form of $\uu$ by $\Bk_\uu$, we get that for $Y \in \mm_\uu$
		\begin{align*}
			\ric^\star(Y,Y) =& -\frac{\Bk(Y,Y)}{2}-\frac{1}{2}\sum_k\Vert\left[Y,U_k\right]_\mm\Vert_g^2-\frac{1}{2}\sum_{\alpha,i} \Vert\left[Y,A_i^\alpha\right]\Vert_g^2+\frac{1}{4}\sum_{j,k}g\left(\left[U_j,U_k\right]_\mm,Y\right)^2 \\
			=&-\frac{\Bk_\uu(Y,Y)}{2}-\frac{1}{2}\sum_k\Vert\left[Y,U_k\right]_\mm\Vert_g^2+\frac{1}{4}\sum_{j,k}g\left(\left[U_j,U_k\right]_\mm,Y\right)^2 \\
			&- \frac{1}{2}\left(\sum_{\alpha,i}\Vert\left[Y,A^\alpha_i\right]\Vert_g^2+\tr\left(\left.\ad^2 Y\right\vert_{V}\right)\right)\\
			=&\ \ric_{U/H}(Y,Y)- \frac{1}{2}\left(\sum_{\alpha,i}\Vert\left[Y,A^\alpha_i\right]\Vert_g^2+\tr\left(\left.\ad^2 Y\right\vert_{V}\right)\right).
		\end{align*}
	\end{proof}
	
	\begin{remark*}\label{remark:Ricci_U in terms of intrinsic and V-action}
		The formula \eqref{eq:Ricci-U formula} for the unimodular Ricci tensor in the $U/H$ direction has the following algebraic interpretation
		\[\ric_g^\star (Y,Y)= \ric_{U/H}(Y,Y)-\Vert \sum_\a S^g\left(\left.\theta( Y)\right|_{V^\a}\right)\Vert_g^2.\]
		Therefore, the unimodular Ricci tensor in the $U/H$ direction approximates the Ricci tensor of $U/H$ with the induced metric precisely when the background metric on the $\Ad(U)$-representations $V^\a$ approaches an $\Ad(U)$-invariant one. Of course, this cannot be achieved in general when $\hat{\zz}\neq0$. However, for the compact representation of $\Ad(K)$ on $V^\a$, this is possible, and indeed we show in the following section a result in this direction. Namely, $\Vert S^g\left(\left.\theta( L_t)\right|_{V^\a}\right)\Vert_{g(t)}^2$ is integrable along a unimodular Ricci flow $g(t)$ starting at a $\theta$-adapted standard metric, for any $g(t)$-unitary vector $L_t \in \ll \subset \kk$.
	\end{remark*}
	
	In the compact direction $K/H$ the formula \eqref{eq:Ricci-U formula} has an even better expression.
	
	\begin{corollary} \label{prop:L bound}
		Let $(M,g)$ be a stable $U\ltimes_\theta V$-homogeneous manifold with a $\theta$-adapted standard metric $g$ and let $\gg=\hh\oplus\mm_\uu \bigoplus_{\a \in \mathcal{I}^*} V^{\a}$ be a reductive decomposition adapted to $g$, where $\mm_\uu=\ll\oplus\hat{\zz}$.
		Then for a $\left.g\right\vert_{V^\a\times V^\a}$-diagonalizing $\langle\cdot,\cdot \rangle$-orthonormal frame, $\{\bar{A}^\alpha_i\}_{i=1}^{d_\a} \subset V^\a$, as in Notation \ref{not:orth frame}, we have that for $L \in \ll \subset \kk$,
		\begin{equation}\label{eq:Ricci-L formula}
			\ric^\star(L,L) = \ric_{U/H}(L,L)-\frac{1}{4}\sum_{\alpha,i,j}\left\langle \left[L,\bar{A}^\a_i\right],\bar{A}^\alpha_j \right\rangle^2\left(\frac{g^\alpha_i}{g^\a_j}+\frac{g^\a_j}{g^\alpha_i}-2\right).
		\end{equation}
	\end{corollary}
	
	\begin{proof}
		This is an immediate consequence of the formula in Proposition \ref{prop:Ricci-U}, since we chose our background metric $\left. \langle \cdot,\cdot \rangle\right|_{V^\a\times V^\a}$ to be $\Ad(K)$-invariant. This implies that for any $L \in \ll$, $\left.\ad L\right|_{V^\a}$ is skew-symmetric with respect to $\left. \langle \cdot,\cdot \rangle\right|_{V^\a\times V^\a}$. Thus,
		\begin{align*}
			\ric^\star(L,L) 
			&=\ric_{U/H}(L,L)- \frac{1}{2}\left(\sum_{\alpha,i}\Vert\left[L,A^\alpha_i\right]\Vert_g^2+\tr\left(\left.\ad^2 L\right\vert_{V}\right)\right)\\
			&=\ric_{U/H}(L,L)- \frac{1}{2}\left(\sum_{\alpha,i,j}g\left(\left[L,A^\alpha_i\right],A^\a_j\right)^2+\sum_{\alpha,i}\left\langle \left[L,\left[L,\bar{A}^\alpha_i\right]\right],\bar{A}^\alpha_i\right\rangle\right) \\
			&=\ric_{U/H}(L,L)- \frac{1}{2}\left(\sum_{\alpha,i,j}\frac{g_j^\a}{g_i^\alpha}\left\langle\left[L,\bar{A}^\alpha_i\right],\bar{A}^\a_j\right\rangle^2-\sum_{\alpha,i}\left\langle \left[L,\bar{A}^\alpha_i\right],\left[L,\bar{A}^\alpha_i\right]\right\rangle\right) \\
			&=\ric_{U/H}(L,L)- \frac{1}{2}\left(\sum_{\alpha,i,j}\frac{g_j^\a}{g_i^\alpha}\left\langle\left[L,\bar{A}^\alpha_i\right],\bar{A}^\a_j\right\rangle^2-\sum_{\alpha,i,j}\left\langle \left[L,\bar{A}^\alpha_i\right],\bar{A}^\a_j\right\rangle^2\right) \\
			&=\ric_{U/H}(L,L)- \frac{1}{4}\sum_{\alpha,i,j}\left\langle \left[L,\bar{A}^\alpha_i\right],\bar{A}^\a_j \right\rangle^2 \left(\frac{g^\alpha_i}{g^\a_j}+\frac{g_j^\a}{g_i^\alpha}-2 \right). 
		\end{align*}
	\end{proof}
	
	Since $\hat{\zz} \subset \zz(\uu)$, we can use the estimates for $\ric_{U/H}$ in the direction for the largest eigenvalue of $\left.g\right|_{\ll\times \ll}$ obtained in the proof of \cite[Theorem 3.1]{boe} to get the following corollary.
	\begin{corollary}\label{cor:Ricci^star on Lm is bounded}
		Let $(M,g)$ be a stable $U\ltimes_\theta V$-homogeneous manifold with a $\theta$-adapted standard metric $g$ and let $\gg=\hh\oplus\mm_\uu \bigoplus_{\a \in \mathcal{I}^*} V^{\a}$ be a reductive decomposition adapted to $g$, where $\mm_\uu=\ll\oplus\hat{\zz}$.
		Let $\{\bar{A}^\alpha_i\}_{i=1}^{d_\a} \subset V^\a$ be a $\left.g\right\vert_{V^\a\times V^\a}$-diagonalizing $\langle\cdot,\cdot \rangle$-orthonormal frame, as in Notation \ref{not:orth frame}.
		If $L_m \in \ll\subset \kk$ is an eigenvector corresponding to the largest eigenvalue $g^\ll_m$ of $\left.g\right|_{\ll\times \ll}$. Then we have the following bound for the unimodular Ricci curvature
		\begin{align*}
			\ric^\star(L_m,L_m) &\geq -\frac{\Bk_\kk\left(L_m,L_m\right)}{4}- \frac{1}{4}\sum_{\alpha,i,j}\left\langle \left[L_m,\bar{A}^\alpha_i\right],\bar{A}^\a_j \right\rangle^2 \left(\frac{g^\alpha_i}{g^\a_j}+\frac{g_j^\a}{g_i^\alpha}-2 \right)\\
			&\geq - \frac{1}{4}\sum_{\alpha,i,j}\left\langle \left[L_m,\bar{A}^\alpha_i\right],\bar{A}^\a_j \right\rangle^2 \left(\frac{g^\alpha_i}{g^\a_j}+\frac{g_j^\a}{g_i^\alpha}-2 \right),
		\end{align*}
		where $\Bk_\kk$ is the Killing form of $\kk$, which is negative semidefinite.
	\end{corollary}
	These special algebraic estimates for the unimodular Ricci tensor follow from the algebraic and geometric compatibility of $\theta$-adapted standard metrics. They serve as the starting point for the analysis of the long-time behavior of unimodular Ricci flows along $\theta$-adapted standard metrics.
	
	\section{The Dynamical Alekseevskii Conjecture on $\theta$-adapted standard metrics}\label{section:main result}
	
	\subsection{On the regularity of the eigenvalues of homogeneous Ricci flows}
	
	Let us quickly discuss the regularity of eigenvalues of a (unimodular) homogeneous Ricci flow solution $g(t)$. We need that in order to use ODE comparison arguments in what follows. 
	
	Notice that if $\phi \colon (a,b) \times \mathbb{R}^k \to \mathbb{R}$ is a $C^1$ map, then by \cite[Lemma B.40]{chow} it follows that 
	\begin{equation}\label{eq:extremal regularity}
		\frac{d^-}{dt}\max_{x\in C} \phi_t(x) =\min_{x\in C_t}\frac{\partial \phi_t}{\partial t}(x),
	\end{equation}
	where $C$ is a compact subset of $\mathbb{R} ^k$ and $C_t \coloneqq \{x_t \in C \ \vert \ \phi_t\left(x_t\right)= \max_{x\in C} \phi_t(x)\}$, and $\frac{d^-}{dt}$ is the left-hand derivative. We can apply this to the homogeneous (unimodular) Ricci flow solution $g(t)$ in order to get control on the maximum eigenvalue of $g(t)$ with respect to a fixed background metric. 
	\begin{remark*}\label{remark:this implies lipschitz}
		We have analogous results as to equation \eqref{eq:extremal regularity} by replacing the maximum for the minimum or replacing the left-hand derivative for the right-hand derivative in equation \eqref{eq:extremal regularity}. By combining these results we can conclude for example that the extremal eigenvalues of $g(t)$ are locally Lipschitz and use comparison arguments on them (see e.g. inequalities \eqref{eq:max eigen is decreasing} and \eqref{eq:min eigen is increasing}).
	\end{remark*}

	We can also adapt this property for other smooth functions defined in terms of the metric $g(t)$ in order to get regularity information on more eigenvalues (see definition \eqref{def:eigenvalues as max} on the next subsection).
	
	In our case, however, even more is true, since the (unimodular) homogeneous Ricci flow is analytic. This follows from the Cauchy-Kovalevskaya theorem (see \cite[Chapter 1, Section D]{fol}) and the observation that the homogeneous Ricci flow is an equation given by rational functions. In this situation, by Rellich's theorem \cite[Chapter 1, $\S$ 1, Theorem 1]{rel}, there exists a choice of analytic eigenvalues with a corresponding analytic orthonormal frame with respect to the fixed background metric. In particular, this implies that any continuous choice of eigenvalues is, up to a finite set of singularities, automatically analytic on a compact interval. This is a remarkable property, but we do not require that much in this paper; for our purposes, the Lipschitz condition is sufficient to apply the fundamental theorem of calculus (see Lemma \ref{lemma:Ricci-A is integrable}).
		\begin{remark*}
		The real analyticity assumption is indeed very strong. Just assuming smoothness of a family $A(t)$ generally leads to a drastic loss of regularity for the eigenvalues. For instance, there are smooth families $A(t)$ that do not admit a $C^2$-choice of eigenvalues (see Example in \cite{km}), and the orthonormal eigenvectors may not even admit a continuous choice (see \cite[Example 2.7]{pr}). For a more comprehensive and up-to-date survey on the interesting topic of the regularity of roots of polynomials, we refer to \cite{pr}.
	\end{remark*}
	
	\subsection{Dynamical estimates and the main theorem}\label{subsec:main}
	
	The a priori bounds obtained in Corollary \ref{cor:Ricci-A and Ricci A-bar} already gives us that the pinching quotient $\frac{g^\a_{d_\a}}{g^\a_{1}}(t)$ for the unimodular Ricci flow solution $g(t)$ restricted to $V^\a$ is nonincreasing for all $ \a \in \mathcal{I}^*$, where we order the eigenvalues as $g_1^\a \leq \ldots \leq g_{d_\a}^\a$. Indeed, using \cite[Lemma B.40]{chow},
	we get that for the maximum eigenvalue $g^\a_{d_\a}$,
	\begin{equation}\label{eq:max eigen is decreasing}
		\frac{d^-g_{d_\a}^\a}{dt}\leq -2\ric^\star\left(\bar{A}_{d_\a}^\a,\bar{A}_{d_\a}^\a\right) = -\sum_{\substack{j, k}} \left\langle \left[U_k, \bar{A}^\a_j\right],\bar{A}_{d_\a}^\a \right\rangle^2\left(\frac{\left(g_{d_\a}^\a\right)^2}{g^\a_j}-g^\a_j\right)\leq 0,
	\end{equation}
	where $\frac{d^-}{dt}$ is the left-hand derivative and in the last inequality we used that $g_{d_\a}^\a \geq g^\a_j$, for $j=1,\ldots,d_\a$. And analogously for the smallest eigenvalue (see Remark \ref{remark:Ricci-A bounds for smallest eigen}), we get
	\begin{equation}\label{eq:min eigen is increasing}
		\frac{d^-g_1^\a}{dt}\geq -2\ric^\star\left(\bar{A}_1^\a,\bar{A}_1^\a\right) = -\sum_{j,k} \left\langle \left[U_k, \bar{A}^\a_j\right],\bar{A}_1^\a \right\rangle^2\left(\frac{\left(g_1^\a\right)^2}{g^\a_j}-g_j^\a \right)\geq 0.
	\end{equation}
	
	We can integrate these relations to conclude that $g^\a_{d_\a}$ is nonincreasing and $g^\a_{1}$ is nondecreasing.
		
	As in \cite[Theorem 3.2]{boe}, this combined with Corollary \ref{prop:L bound} implies that if $K/H$ is not a torus, then when the initial metric $\left.g_0\right|_{V^\a\times V^\a}$ on each $V^\a$ is sufficiently pinched, there exits $\epsilon>0$ such that $\ric_{g(t)}^\star\left(L_{m'}(t),L_{m'}(t)\right)\geq \frac{\epsilon}{g^\ll_{m'}(t)}>0$, for all $t \in [0,T)$, where $L_{m'}$ is the eigenvector with largest eigenvalue $g^\ll_{m'}$ in the direction of the semisimple part of $\ll$. Which in turn, implies that the flow blows up in finite time.
	
	These algebraic estimates we have for the unimodular Ricci curvature in the direction of each $V^\a$ are, however, not enough to conclude that $\frac{g^\a_{d_\a}}{g^\a_{1}}(t)$ converges to $1$, as $t \to \infty$ on an immortal flow. 
	Instead, in what follows we obtain integral estimates which will be enough to confirm the dynamical Alekseevskii conjecture on homogeneous Ricci flows along $\theta$-adapted standard metrics.	
	
	In the lemmas that follow we assume, to simplify the notation, that $\mathcal{I}=\{\alpha_1\}$, i.e., we have a single submodule $V = V^{\alpha_1}$ such that $\hat{\zz}$ acts as a multiple of the identity plus a $\langle \cdot, \cdot \rangle$-skew-symmetric map. By the formulas in Proposition \ref{prop:A bounds} and Proposition \ref{prop:Ricci-U}, it is easy to see that the following lemmas are still true for the more general Ricci flow invariant condition $V^{\a_1} \perp_{g(t)} \ldots \perp_{g(t)} V^{\a_p}$. 
	
	For the next lemma, note that we can define the largest eigenvalue $g^V_d$ of the metric $\left.g\right|_{V\times V}$ with respect to the background metric $\langle\cdot,\cdot\rangle$ as $g^V_d\coloneqq \max \{g(X,X) \ | \ X\in V, \ \langle X,X\rangle=1\}$ and since the logarithm is strictly increasing we can define $\log g^V_d\coloneqq \max\log \{g(X,X) \ | \ X\in V, \ \langle X,X\rangle=1\}$. 
	
	Analogously, we can define the sum of the $d-j+1$-largest eigenvalues $\sum_{i=j}^dg^V_i$ as 
	\begin{align}\label{def:eigenvalues as max}
		\sum_{i=j}^dg^V_i\coloneqq \max \{\sum_{i=1}^{d-j+1}g(X_i,X_i) \ | \ &X_1,  \ldots,X_{d-j+1}\in V, \\ \notag
		\langle &X_i,X_k \rangle=0, i\neq k, \ \langle X_i,X_i\rangle=1\}.
	\end{align}
	And thus we can apply equality \eqref{eq:extremal regularity} to this geometric quantities along the unimodular Ricci flow.
	\begin{lemma}\label{lemma:Ricci-A is integrable}
		Let $(M,g(t))$, $t \in [t_0,T)$, be a stable $U\ltimes_\theta V$-homogeneous manifold with a maximal  unimodular Ricci flow solution along $\theta$-adapted standard metrics $g(t)$, and let $\gg=\hh\oplus\mm_\uu \bigoplus_{\a \in \mathcal{I}^*} V^{\a}$ be a reductive decomposition adapted to $g(t)$. 
		Then, for any $i_0 \leq d$, the nonnegative quantity 
		\[\max_{\{\bar{A}_i(t)\}^d_{i=i_0} \in C_{i_0,t}} \left(\sum_{i=i_0} ^{d}\ric_{g(t)} ^\star (A_i(t),A_i(t)) \right)\]
		is bounded from above by an integrable quantity along the flow, where $C_{i_0,t}$ is the set of $\langle\cdot,\cdot\rangle$-orthonormal eigenvectors of $\left.g(t)\right|_{V\times V}$ corresponding to the eigenvalues $g^V_{i_0}(t) \leq \ldots\leq g^V_d(t)$, and $A_i(t) = \frac{\bar{A}_i}{\sqrt{g^V_i(t)}}$, as in Notation \ref{not:orth frame}.
	\end{lemma}
	\begin{proof}
		Let $\langle \cdot, \cdot \rangle$ be an $\Ad(K)$-invariant background metric and let us order the eigenvalues of $g(t)$ restricted to $V$, $g^V_1 \leq \ldots \leq g^V_d$. 
		
		By equation \eqref{eq:extremal regularity} and the inequality \eqref{ineq: Ricci-A-bar} we get that for any $i_0 \leq d$,
		\begin{align*}
			\frac{d^-}{dt}\left(\sum_{i=i_0}^d g^V_i(t)\right)=& \min_{\{\bar{A}_i(t)\}^d_{i=i_0} \in C_{i_0,t}} \left(-2\sum_{i=i_0} ^{d}\ric_{g(t)} ^\star \left(\bar{A}_i(t),\bar{A}_i(t)\right) \right)\\
			=&  -2\max_{\{\bar{A}_i(t)\}^d_{i=i_0} \in C_{i_0,t}} \left(\sum_{i=i_0} ^{d}\ric_{g(t)} ^\star \left(\bar{A}_i(t),\bar{A}_i(t)\right) \right)\\
			\le&\ 0,
		\end{align*}
		and thus $\sum_{i=i_0}^dg^V_i(t)$ is nonincreasing.
	
		Let us set $\bar{f}_{i_0}(t)\coloneqq\max_{\{\bar{A}_i(t)\}^d_{i=i_0} \in C_{i_0,t}} \left(\sum_{i=i_0} ^{d}\ric_{g(t)} ^\star (\bar{A}_i(t),\bar{A}_i(t)) \right) \geq0$. By Remark \ref{remark:this implies lipschitz}, we have that $\sum_{i=i_0}^d g^V_i(t)$ is Lipschitz on compact intervals. Thus, by the fundamental theorem of calculus (\cite[Theorem 7.20]{rud})), we get that
		\begin{align*}
			-\sum_{i=i_0}^d g^V_i(t_0) \leq& \sum_{i=i_0}^d g^V_i(t)-\sum_{i=i_0}^d g^V_i(t_0)\\
			=& -2\int^t_{t_0}\bar{f}_{i_0}(s)ds\leq 0.
		\end{align*}
		Thus,
		\[0\leq \lim_{t\to T}\int^t_{t_0}\bar{f}_{i_0}(s)ds \leq \sum_{i=i_0}^d g^V_i(t_0)<\infty,\]
		where $T$ is the maximal time of existence of $g(t)$. 
		
		Note that, by estimate \eqref{ineq: Ricci-A-bar},
			\begin{align*}
			0\leq 2\sum_{i=i_0}^{d} \ric^\star \left(A_i,A_i\right) =& \sum_{\substack{k}}\sum_{\substack{i\geq i_0\\ j < i_0}}\left\langle \left[U_k, \bar{A}_i\right],\bar{A}_j \right\rangle^2\left(\frac{g_i}{g_j}-\frac{g_j}{g_i}\right)\\
		=& \sum_{\substack{k}} \sum_{\substack{i\geq i_0}}\frac{1}{g_i} \sum_{j< i_0}\left\langle \left[U_k, \bar{A}_i\right],\bar{A}_j \right\rangle^2\left(\frac{(g_i)^2}{g_j}-g_j\right)\\
		\leq&\ \frac{1}{g_1}\sum_{\substack{k}} \sum_{\substack{i\geq i_0\\ j < i_0}} \left\langle \left[U_k, \bar{A}_i\right],\bar{A}_j \right\rangle^2\left(\frac{(g_i)^2}{g_j}-g_j\right)\\
		\leq&\ \frac{1}{g_1}\sum_{i=i_0}^{d_\a} \ric^\star \left(\bar{A}_i,\bar{A}_i\right).
	\end{align*}
	
		Using the above facts and that $g^V_1(t) \geq g^V_1(t_0)>0$ (see inequality \eqref{eq:min eigen is increasing}), we get, for all $i_0=1, \ldots, d$, that
		\begin{align*}
			&\int^T_{t_0}\max_{\{\bar{A}_i(s)\}^d_{i=i_0} \in C_{i_0,s}} \left(\sum_{i=i_0} ^{d}\ric_{g_s} ^\star \left(A_i(s),A_i(s)\right) \right)ds \\
			&\leq\int^T_{t_0}\max_{\{\bar{A}_i(s)\}^d_{i=i_0} \in C_{i_0,s}} \left(\frac{1}{g^V_1(s)}\sum_{i=i_0} ^{d}\ric_{g_s} ^\star \left(\bar{A}_i(s),\bar{A}_i(s)\right) \right)ds \\
			&\leq \frac{1}{g^V_1(t_0)}\int^T_{t_0}\bar{f}_{i_0}(s) ds <\infty,
		\end{align*}
		where, recall, we denote $\bar{f}_{i_0}(s) = \max_{\{\bar{A}_i(s)\}^d_{i=i_0} \in C_{i_0,s}} \left(\sum_{i=i_0} ^{d}\ric_{g_s} ^\star \left(\bar{A}_i(s),\bar{A}_i(s)\right) \right)$.
	\end{proof}
	
	The integrable bound given above is going to be crucial for us because with it we have control over the error term appearing in the formula for $\ric_g^\star (Y,Y)$, $Y \in \mm_\uu$, in Proposition \ref{prop:Ricci-U}. This is the content of the following corollary of Lemma \ref{lemma:Ricci-A is integrable}.
	
	\begin{corollary}\label{cor:bad term is bounded}
		Let $(M,g(t))$, $t \in [t_0,T)$, be a stable $U\ltimes_\theta V$-homogeneous manifold with a maximal unimodular Ricci flow solution along $\theta$-adapted standard metrics $g(t)$, and let $\gg=\hh\oplus\mm_\uu \bigoplus_{\a \in \mathcal{I}^*} V^{\a}$ be a reductive decomposition adapted to $g(t)$. 
		Let $\{U_k(t)\}_{k=1}^{\dim \mm_\uu}$ be a $g(t)$-orthonormal frame for $\mm_\uu$ and $\{\bar{A}_i\}_{i=1}^d$ be a $\left.g(t)\right|_{V\times V}$-diagonalizing $\langle\cdot,\cdot\rangle$-orthonormal frame for $V$, with corresponding eigenvalues $g^V_i(t)$, $i=1,\ldots,d$ (see Notation \ref{not:orth frame}). Then, the nonnegative sum
		\[\sum_{k=1}^{\dim \mm_\uu}\sum_{i,j=1}^d\left\langle \left[U_k(t), \bar{A}_i(t)\right],\bar{A}_j(t) \right\rangle^2\left(\frac{g^V_i(t)}{g^V_j(t)}+\frac{g^V_j(t)}{g^V_i(t)}-2\right)\]
		is bounded from above by an integrable quantity along the flow.
	\end{corollary}
	\begin{proof}
		Recall formula \eqref{ineq:Ricci-A is pos},
		\[\sum_{i=i_0}^d \ric^\star \left(A_i,A_i\right) = \sum_{\substack{k}}\sum_{\substack{i\geq i_0\\ j < i_0}}\left\langle \left[U_k, \bar{A}_i\right],\bar{A}_j \right\rangle^2\left(\frac{g^V_i}{g^V_j}-\frac{g^V_j}{g^V_i}\right),\]
		for all $i_0 = 1 \ldots d$. And observe that if $g^V_i(t)\geq g^V_j(t)$, then
		\[ \left\langle \left[U_k(t), \bar{A}_i(t)\right],\bar{A}_j(t) \right\rangle^2\left(\frac{g^V_i(t)}{g^V_j(t)}+\frac{g^V_j(t)}{g^V_i(t)}-2\right)\leq \left\langle \left[U_k(t), \bar{A}_i(t)\right],\bar{A}_j(t) \right\rangle^2\left(\frac{g^V_i(t)}{g^V_j(t)}-\frac{g^V_j(t)}{g^V_i(t)}\right).\]
		
		It then follows immediately from Lemma \ref{lemma:Ricci-A is integrable} that for all $i,j = 1, \ldots, d$, the nonnegative term
		\[\sum_{k=1}^{\dim \mm_\uu}\left\langle \left[U_k(t), \bar{A}_i(t)\right],\bar{A}_j(t) \right\rangle^2\left(\frac{g^V_i(t)}{g^V_j(t)}+\frac{g^V_j(t)}{g^V_i(t)}-2\right)\]
		is bounded from above by the integrable quantity 
		\begin{equation*}\label{eq:integrable upper bounds}
			f_j(t) \coloneqq \max_{\{\bar{A}_i(t)\}^d_{i=j} \in C_{j,t}} \left(\sum_{i=j} ^{d}\ric_{g(t)} ^\star \left(A_i(t),A_i(t)\right) \right),
		\end{equation*}
		for arbitrary $g(t)$-orthonormal frame $\{U_k(t)\}_{k=1}^{\dim\mm_\uu}$ for $\mm_\uu$, and for $g(t)$-diagonalizing $\langle\cdot,\cdot\rangle$-orthonormal vectors $\bar{A}_i(t)$ and $\bar{A}_j(t)$.
		
		Since these are all finite sums, for some large $N \in \mathbb{N}$,
		\[\sum_{k=1}^{\dim \mm_\uu}\sum_{i,j=1}^d\left\langle \left[U_k(t), \bar{A}_i(t)\right],\bar{A}_j(t) \right\rangle^2\left(\frac{g^V_i(t)}{g^V_j(t)}+\frac{g^V_j(t)}{g^V_i(t)}-2\right)\leq N \sum_{i=1}^d f_i(t),\]
		where $F(t)\coloneqq N \sum_{i=1}^d f_i(t)$  is an integrable quantity along the flow. 
	\end{proof}
	
	By Corollary \ref{cor:Ricci^star on Lm is bounded}, we have that the unimodular Ricci tensor in the direction of the $g(t)$-unitary eigenvector $L_m(t) \in \ll\subset \kk$ 
	corresponding to the largest eigenvalue $g^\ll_m(t)$ of $\left.g(t)\right|_{\ll\times \ll}$ has the following bound
	\begin{align*}
		-2\ric_{g(t)}^\star(L_m(t),L_m(t)) =& -2\ric_{(U/H,g(t))}(L_m(t),L_m(t))\\
		&+\frac{1}{2}\sum_{i,j}\left\langle \left[L_m(t),\bar{A}_i(t)\right],\bar{A}_j(t) \right\rangle^2\left(\frac{g^V_i(t)}{g^V_j(t)}+\frac{g^V_j(t)}{g^V_i(t)}-2\right) \\
		\leq& \ \frac{1}{2}\sum_{i,j}\left\langle \left[L_m(t),\bar{A}_i(t)\right],\bar{A}_j (t)\right\rangle^2\left(\frac{g^V_i(t)}{g^V_j(t)}+\frac{g^V_j(t)}{g^V_i(t)}-2\right).
	\end{align*}
	Combining this with the integral bound in Corollary \ref{cor:bad term is bounded}, we have that
	\begin{align*}
		\frac{d^-}{dt}\log g^\ll_m(t) =& \min_{\{\bar{L}_m(t) \in C_t\}}\left( -2 \ric_{g(t)}^\star(L(t)_m,L(t)_m)\right)\\
		\leq& \  \frac{1}{2}\sum_{i,j}\left\langle \left[L_m(t),\bar{A}_i(t)\right],\bar{A}_j (t)\right\rangle^2\left(\frac{g^V_i(t)}{g^V_j(t)}+\frac{g^V_j(t)}{g^V_i(t)}-2\right)\\
		\leq& \ F(t),
	\end{align*}
	where $F(t)\coloneqq N \sum_{i=1}^d f_i(t)$, for some large $N \in \mathbb{N}$, is integrable along the unimodular Ricci flow.
	
	Integrating both sides, this very rough upper bound gives us
	\begin{align*}
		\log g^\ll_m(t)-\log g^\ll_m(t_0)	\leq \int^T_{t_0} F(s)ds= C_0<\infty.
	\end{align*}
	
	We have then shown that the integral estimates obtained above yield the following lemma.
	\begin{lemma}\label{lemma: gL is bounded}
			Let $(M,g(t))$, $t \in [t_0,T)$, be a stable $U\ltimes_\theta V$-homogeneous manifold with a maximal unimodular Ricci flow solution along $\theta$-adapted standard metrics $g(t)$, and let $\gg=\hh\oplus\mm_\uu \bigoplus_{\a \in \mathcal{I}^*} V^{\a}$ be a reductive decomposition adapted to $g(t)$. 
		Then, the metric induced on the compact submanifold $K/H$ remains bounded along the flow; that is, there exists a positive constant $C_0$ such that
		\[\vert \left.g(t)\right|_{\ll\times\ll} \vert _{g_{t_0}} \leq C_0,\]
		for all $t_0 \leq t <T$.
	\end{lemma}
	\begin{remark*}
		By Lemma \ref{lemma: gL is bounded}, we have that in the directions $\ll$ and $V$ the expansion of the metric $g(t)$ is bounded from above. However, note that in the direction of $\ll^{\perp_{g(t)}} \subset \mm_\uu$ we do not have an upper bound for the expansion of the metric $g(t)$. For instance, let $\bar{Z} \in \hat{\zz}$ be such that $\left.\theta(\bar{Z})\right|_{V^\a}=\a(\bar{Z})\cdot\Id$. Since we assumed that $\alpha(\bar{Z})\neq0$ for some $\alpha \in \uu^*$, it follows from the formula in Remark \ref{remark:Ricci_U in terms of intrinsic and V-action} that the metric in the $\bar{Z}$ direction expands linearly with constant $2\dim V^\alpha\lambda(\bar{Z})^2$.
	\end{remark*}

Before proving the main result of this section, we present a brief application of the integrability of the Ricci curvature along the flow to the case of preflat solvmanifolds (Lemma~\ref{lemma:Ricci-A is integrable}). We say that a Riemannian solvmanifold is \textit{preflat} if it admits a flat left-invariant metric. By \cite[Theorem 1.5]{mil}, we know that a preflat solvmanifold is a stable $U\ltimes_\theta V$-homogeneous manifold, so we can consider the Ricci flow along $\theta$-adapted standard metrics on such spaces.
\begin{corollary}\label{cor:pre flat implies flat conv}
			Let $(\mathcal{S},g(t))$, $t \in (0,\infty)$, be a homogeneous Ricci flow solution on a preflat solvmanifold with an initial $\theta$-adapted standard left-invariant metric. Then the parabolic blowdown $s^{-1}g(st)$ converges in the $C^\infty$ pointed topology to the flat metric as $s \to \infty$.
		\end{corollary}
	\begin{proof}
	By \cite[Theorem 1.5]{mil}, we know that $\mathfrak{s} = \Lie(\mathcal{S})$ has the following decomposition
	\[\mathfrak{s} = \aa \ltimes_{\theta} V, \]
	where, $V$ is the abelian nilradical, $[\aa,\aa]=0$, and $\theta$ is a representation via semisimple operators with purely imaginary eigenvalues. If we then assume that the initial metric $g_0(\aa,V)=0$, then Corollary \ref{cor:bad term is bounded} tells us that the scalar curvature $\Rscal(g(t))$ is integrable along the flow and this implies in turn that $\Rscal(g(t))t \to 0$ as $t \to \infty$. Indeed, otherwise there would be $c>0$ and an infinite sequence $t_i\to \infty$ such that $ \Rscal(g(t_i))t_i = c$. Given $0<\epsilon<1$, we can then pick $t_{i_n}$ such that $\frac{t_{(i_{n}+1)}}{t_{i_{(n+1)}}} \leq \epsilon^n$, for $1 \leq n$. 

	Recall the evolution equation of the scalar curvature along a homogeneous solution to the Ricci flow \cite[Lemma 2.7]{chow},
	\begin{equation}\label{eq:scalar curvature evolution}
		\frac {d}{dt}\Rscal(g(t)) = \Vert \Ric_{g(t)} \Vert^2_{g(t)} \geq 0.
	\end{equation}
	In our case, since $-\Rscal(g(t))$ is always positive and decreasing along the flow, we would then have that 
	\[\int_{t_0}^\infty -\Rscal(g(t)) \geq c\sum_{i=1}^{i_N}\frac{\left(t_i-t_{i-1}\right)}{t_i}\geq c\sum_{n=1}^N\frac{t_{(i_{n+1})}-t_{i_{n}+1}}{t_{(i_{n+1})}} \geq cN - c\sum_{n=1}^N\epsilon^n \xrightarrow[N \to \infty]{} +\infty,\]
	which is a contradiction.
	
	Now, by Hamilton's compactness \cite{ham95}, $s^{-1}g(st)$ converges to a homogeneous solution $g_{\infty}(t)$, $t \in (0,\infty)$, with $\Rscal(g_\infty(t))\leq 0$ and $\Rscal(g_\infty(1)) =0$, which implies by the evolution equation \eqref{eq:scalar curvature evolution} that $g_{\infty}(1)$ is Ricci flat, and thus flat by \cite{ak}.
	\end{proof}

	With Lemma \ref{lemma: gL is bounded} we are ready to prove Theorem \ref{thm:A}, which we rephrase here for convenience.
	\begin{theorem}[Theorem \ref{thm:A}] \label{thm:main}
		Let $(M,g_0)$ be a stable $U\ltimes_\theta V$-homogeneous manifold with a $\theta$-adapted standard metric $g_0$.
		If the universal cover of $M$ is not contractible, then the Ricci flow solution starting at $g_0$ has finite extinction time.
	\end{theorem}
	
	\begin{proof}
		Let $g(t)$ be the unimodular Ricci flow solution such that $g(0)=g_0$.
		By Corollary \ref{cor:uni RF is equivalent to RF}, the unimodular Ricci flow \eqref{eq:unimodular RF} is equivalent to the homogeneous Ricci flow, so we just need to show that $g(t)$ has finite extinction time.
		Moreover, we have shown that the space of $\theta$-adapted standard metrics $\mathfrak{S}_{U\ltimes V}$ is invariant under the unimodular Ricci flow (Corollary \ref{cor: polar is uni RF invariant}). So $g(t) \in \mathfrak{S}_{U\ltimes V}$ for all $t$ in the maximal interval of future existence $[0,T)$, with same adapted reductive decomposition $\gg=\hh\oplus\mm_\uu \bigoplus_{\a \in \mathcal{I}^*} V^{\a}$, $\mathcal{I}^*=\{\a_1,\ldots,\a_p\} \subset \uu^*$.
		
		Let us consider the unimodular Ricci tensor formula in the compact, semisimple direction $\ll_{ss} \subset\ll \subset \kk$, namely, $\ll_{ss}= \ll\cap\kk_{ss}$, where $\kk_{ss}$ is the derived, compact, semisimple Lie algebra $\kk_{ss}\coloneqq[\kk,\kk] \subset \kk$. Since, by hypothesis, the universal cover of $M$ is not contractible, $\ll_{ss}$ is not trivial and it corresponds to the tangent space at the base point of the nonflat compact submanifold $K_{ss}/\left(K_{ss}\cap H\right) \subset U/H$. 
		
		Recall the formula \eqref{eq:Ricci-U formula} for the unimodular Ricci curvature of a $\theta$-adapted standard metric in the direction of $Y \in \mm_\uu$,
		\[\ric^\star(Y,Y) =\ric_{U/H}(Y,Y)- \frac{1}{2}\left(\sum_{\alpha,i}\Vert\left[Y,A^\alpha_i\right]\Vert_g^2+\left.\tr\left(\ad^2Y\right\vert_{V}\right)\right),\]
		and note that the second term in the right-hand side splits nicely with respect to the decomposition $V^{\a_1} \perp_{g(t)} \ldots \perp_{g(t)} V^{\a_p}$. Our goal is to control this error term and show that in the direction of $\ll_{ss}$ the metric $g(t)$ eventually shrinks, collapsing in finite time. As we observed before, Lemmas \ref{lemma:Ricci-A is integrable} and \ref{lemma: gL is bounded} generalize immediately for the case $V^{\a_1} \perp_{g(t)} \ldots \perp_{g(t)} V^{\a_p}$; thus, we can assume without loss of generality that $\mathcal{I}=\{\alpha_1\}$, i.e., that we have a single submodule $V = V^{\alpha_1}$, such that $\hat{\zz}$ acts on $V$ as a multiple of the identity plus a semisimple operator with purely imaginary eigenvalues.  
		
		Let us fix an $\Ad(K)$-invariant background metric $\langle\cdot,\cdot\rangle$ on $\mm$ and a $\left.g(t)\right|_{V\times V}$-diagonalizing $\langle\cdot,\cdot\rangle$-orthonormal frame $\{\bar{A}_i\}_{i=1}^d$, with corresponding eigenvalues $g^V_1, \ldots, g^V_d$, as in Notation \ref{not:orth frame}. Let $\bar{L}_{m'}(t)$ be a $\langle \cdot, \cdot \rangle$-unitary eigenvector corresponding to the largest eigenvalue $g^\ll_{m'}(t)$ of $P_t^{\ll_{ss}} \coloneqq \pr_{\ll_{ss}}\circ \left.P_t\right| _{\ll_{ss}}$. 
		
		Let us now recall the formula \eqref{eq:Ricci-L formula} for the unimodular Ricci curvature of a $\theta$-adapted standard metric in the direction of $L \in \ll$,
		\[	\ric^\star(L,L) = \ric_{U/H}(L,L)-\frac{1}{4}\sum_{i,j}\left\langle \left[L,\bar{A}_i\right],\bar{A}_j \right\rangle^2\left(\frac{g^V_i}{g^V_j}+\frac{g^V_j}{g^V_i}-2\right).\]
		The computation in the proof of \cite[Theorem 3.1]{boe} shows that
		\[\ric_{U/H}\left(\bar{L}_{m'},\bar{L}_{m'}\right) \geq \frac{b_0}{4}>0,\]
		where $b_0\coloneqq \min \{-\Bk_\kk(L,L) \ | \ L \in \ll_{ss} \ | \ \langle L, L \rangle =1\}>0$.
		
		Let $L_{m'}\coloneqq\frac{\bar{L}_{m'}(t)}{\sqrt{g^\ll_{m'(t)}}}$. By Corollary \ref{cor:bad term is bounded},  
		\[\sum_{i,j}\left\langle \left[L_{m'}(t), \bar{A}_i(t)\right],\bar{A}_j(t) \right\rangle^2\left(\frac{g^V_i(t)}{g^V_j(t)}+\frac{g^V_j(t)}{g^V_i(t)}-2\right)\]
		is bounded from above by a function $F(t)$ which is integrable along the flow.
		
		Therefore, we get that
		\begin{align*}
			\frac{d^-g^\ll_{m'}(t)}{dt} \leq &-2 \ric_{g(t)}^\star\left(\bar{L}_{m'}(t),\bar{L}_{m'}(t)\right) \\
			= &-2 \ric_{(U/H,g(t))}\left(\bar{L}_{m'}(t),\bar{L}_{m'}(t)\right)\\
			&+\frac{1}{2}\sum_{i,j}\left\langle \left[\bar{L}(t)_{m'},\bar{A}(t)_i\right],\bar{A}(t)_j \right\rangle^2\left(\frac{g^V_i(t)}{g^V_j(t)}+\frac{g^V_j(t)}{g^V_i(t)}-2\right) \\
			\leq &-\frac{b_0}{2}+\frac{g^\ll_{m'}(t)}{2}\sum_{i,j}\left\langle \left[L_{m'}(t),\bar{A}_i(t)\right],\bar{A}_j(t) \right\rangle^2\left(\frac{g^V_i(t)}{g^V_j(t)}+\frac{g^V_j(t)}{g^V_i(t)}-2\right).
		\end{align*}
		
		By Lemma \ref{lemma: gL is bounded}, we have a uniform upper bound $C_0$ only depending on $g_0$ such that $g^\ll_{m'}(t) \leq C_0$ for all $t \in [0,T)$, we then get that
		\begin{align*}
			g^\ll_{m'}(t) -g^\ll_{m'}(0) \leq& -\frac{b_0}{2}t+\int^t_{0}\frac{C_0}{2}F(s)ds\\
			\leq& -\frac{b_0}{2}t + \tilde{C}_0,
		\end{align*}
		where $\tilde{C}_0$ is an upper bound that only depends on $g_0$. Thus, $g^\ll_{m'}(t)$ goes to zero in finite time.
	\end{proof}
	
	\section{The set of $U\ltimes V$-invariant positive scalar curvature metrics}\label{section:connectedness}
			A consequence of the dynamical Alekseevskii conjecture would be that the set $\mathcal{M}^G_{psc}$ of $G$-invariant positive scalar metrics on $G/H$ is contractible. This is so because of the following nice argument that we owe to Christoph Böhm.
			\begin{proposition}\label{prop:contractability of global attract}
				Let $\mathcal{M}^G_{psc}$ be the space of $G$-invariant positive scalar metrics on $G/H$. If $\mathcal{M}^G_{psc}$ is a global attractor for the Ricci flow, then $\mathcal{M}^G_{psc}$ is contractible.
			\end{proposition}
			\begin{proof}
				Let us first consider a bump function $b$ on the space of $G$-invariant metrics on $G/H$, supported on $\mathcal{M}^G_{\Rscal <1}$, namely, the $G$-invariant metrics with scalar curvature smaller than 1. Let us modify the Ricci tensor as follows,
				\[\widetilde{\ric}(g) \coloneqq b(g) \ric(g).\]
				By the results in \cite{laf}, a homogeneous Ricci flow solution has finite extinction time if and only if the scalar curvature, $\Rscal(g_t)$, blows up. Thus, $\widetilde{\ric}$ is such that its flow $\phi$ exists for all $t\geq0$, for all $g \in \mathcal{M}^G$. Moreover, since we assumed that $\mathcal{M}^G_{psc}$ is a global attractor, for any compact subset $K \subset \mathcal{M}^G$ the entrance time $T_K \coloneqq \inf \{t >0 \ \vert \ \phi_t(g) \in \mathcal{M}^G_{psc} , \forall g \in K \}$ is finite by the continuous dependence of $\phi$ on the initial conditions.
				
				Now, since $\mathcal{M}^G$ is contractible, every sphere map $f\colon \mathbb{S}^d \to \mathcal{M}^G_{psc}$ has a homotopy $h\colon \mathbb{S}^d\times [0,1]\to \mathcal{M}^G$, $h_0=f$, to a constant function $h_1\colon \mathbb{S}^d\to \mathcal{M}^G$. Let $T$ be the entrance time of $h(\mathbb{S}^d\times [0,1])$. The map $\tilde{h} \colon \mathbb{S}^d \times [0,T+1]  \to \mathcal{M}^G_{psc}$,
				\begin{align*}
					\tilde{h}_t(g) = \begin{cases*}
					 \phi_t\circ h_0(g), t \in [0,T], \\
					 \phi_T\circ h_{t-T}(g), t \in [T,T+1]
					\end{cases*}
				\end{align*}
				is a homotopy between $f$ and the constant function $h_1$, which shows that every homotopy group of $ \mathcal{M}^G_{psc}$ is trivial. Since this set is an open subset of the smooth manifold $\mathcal{M}^G$, it is in particular a CW-complex and thus contractible by \cite[Theorem 4.5]{hat}.
			\end{proof}
			
			The topology of the moduli space of Riemannian metrics satisfying a given curvature condition has long been a topic of interest in differential geometry and geometric analysis. This is especially due to the connection between topology and the existence of critical points of a given functional, as well as its relevance in deformation theory. The particular case of positive scalar curvature has been studied extensively over the past decades and remains an active area of research, with many intriguing open questions (see, e.g., \cite{ew}, \cite{rost}, \cite{ks}).

			By \cite[Theorem 3.2]{boe}, the dynamical Alekseevskii conjecture holds for coverings of an arbitrary compact homogeneous Riemannian manifold $K/H$. Hence, by Proposition \ref{prop:contractability of global attract}, we have that $\mathcal{M}^K_{psc}$ is contractible. However, there are examples of closed manifolds $M$ for which the moduli space of all positive scalar curvature metrics modulo the action via pullbacks by diffeomorphisms, $\mathcal{R}^+(M)$, is disconnected. In fact, Kreck and Stolz showed in \cite{ks} that a fixed closed smooth manifold $M$ may admit $G$-homogeneous positive scalar curvature metrics, for different $G$-actions, lying in distinct connected components of $\mathcal{R}^+(M)$. On the other hand, $\mathcal{R}^+(\mathbb{S}^2)$ is contractible \cite[Theorem 3.4]{rost}. Moreover, by \cite{mar}, the space $\mathcal{R}^+(M^3)$ is path-connected for any orientable closed 3-manifold, and the proof uses the Ricci flow. In more recent work, Bamler and Kleiner \cite{bakle} have improved this result, showing that $\mathcal{R}^+(M^3)$ is either empty or contractible. Finally, it is known that the connected component of the round metric on the sphere $\mathbb{S}^d$, $d\geq 2$, has an abelian fundamental group \cite[Corollary 5.2]{wal}.
			
			Even though we do not have the confirmation of the conjecture for the whole set of $U\ltimes V$-invariant metrics on stable $U\ltimes_\theta V$-homogeneous manifolds, we prove in this section that the set of $U\ltimes V$-invariant positive scalar curvature metrics is indeed contractible.
		
			A classic result on Riemannian submersions is the following (see {\cite[9.12]{be}}).
		\begin{proposition}\label{prop:G-Riemannian sub}
			Let $(M, g)$ be a Riemannian manifold and $G$ a closed subgroup of the isometry group of $(M, g)$. Assume that the projection $\pi$ from $M$ to the quotient space $M/G$ is a smooth submersion. Then there exists one and only one Riemannian metric $g^B$ on $B = M/G$ such that $\pi$ is a Riemannian submersion.
		\end{proposition}
	
		With this proposition in hand, we now study our $U\ltimes V$-homogeneous manifold as the total space of the Riemannian submersion induced by the isometric action of the nilradical $V$. The first thing we observe is that in this case the $V$-action induces a homogeneous Riemannian submersion.
		\begin{lemma}\label{lemma:Homo Riem sub}
			Let $(M =G/H,g)$ be a homogeneous Riemannian manifold, with $G=U\ltimes N$ and $H < U$. Then $N$ acts properly, freely and isometrically on $(M,g)$, thus $M/N$ is a smooth manifold and $\pi \colon M \to M/N$ is a $U$-equivariant smooth submersion. Moreover, the unique Riemannian submersion metric $g^B$ induced by $\pi$ on $M/N \cong U/H$ is $U$-invariant.
		\end{lemma}
		\begin{proof}
			Since $N$ is a closed subgroup of $G$, we only need to show that it acts freely on $M$. This follows immediately from the fact $N$ is a normal subgroup of $G$ disjoint from $H$. Indeed, let $n \in N$ such that $npH = pH$. Then $p^{-1}np \in H\cap N =\{e\}$. By the quotient manifold theorem, we have that $\pi \colon M \to M/N$ is a Riemannian submersion and, by Proposition \ref{prop:G-Riemannian sub}, this submersion induces a unique metric $g^B$ on $B\coloneqq M/N$. 
			
			Notice that $B=M/N = N\backslash G/H \cong U/H$ and if $p\in G$ let us denote the class $NpH$ as $[p] \coloneqq NpH = \pi (pH)$. Let $u \in U$ and $X \in T_{[p]}B$, we claim that
			\[\left((L_u)_*X\right)^\ho = (L_u)_*X^\ho,\]
			that is, the horizontal lift of the push-forward $\left((L_u)_*X\right)$ is the push-forward by $(L_u)_*$ of the horizontal lift of $X$. This follows from the fact $N\lhd G$, because then $\pi\circ L_u = L_u \circ \pi $ and $\left(L_p\right)_*T_{eH}N =  T_{pH}N \subset G/H$, thus $\left(L_p\right)_*\ho_{eH} = \ho_{pH}$.
			
			From this we get immediately the result. Indeed let $u,w \in U$ and $X,Y \in T_{[w]}B$, then 
			\begin{align*}
				L_u^*g^B\left(X,Y\right)_{[w]} \coloneqq& \ g^B\left((L_u)_*X,(L_u)_*Y\right)_{[uw]}\\
				=& \ g\left((L_u)_*X^\ho,(L_u)_*Y^\ho\right)_{uwH}\\
				=& \ g\left(X^\ho,Y^\ho\right)_{wH}\\
				\eqqcolon& \ g^B\left(X,Y\right)_{[w]}.
			\end{align*}
		\end{proof}
		
		A consequence of Lemma \ref{lemma:Homo Riem sub} is that we can choose a reductive complement $\mm = \mm_\uu\oplus\nn \cong T_pM$, with $\mm_\uu \subset \uu$, such that 
		$\left.\pi_*\right\vert_{\mm_\uu}= \Id_{\mm_\uu}$ and $\mm_\uu \cong T_{[p]}B$.
			Moreover, we see that given a $U$-invariant metric $g^B$ on $U/H$, an $\Ad(H)$-invariant inner product $g^F$ on $\nn$, and a $U\ltimes N$-invariant horizontal distribution $\mathcal{H}$ (equivalently $\mathcal{H}_p$ must be $\Ad(H)$-invariant) we can construct a $U\ltimes V$-invariant metric $g=g\left(g^B,g^F, \mathcal{H}\right)$ on $U\ltimes V$.
		
		Another simple observation is that if you have two distinct $U\ltimes N$-invariant horizontal distributions $\mathcal{H}^1$ and $\mathcal{H}^2$, then we can uniquely determine $\mathcal{H}^2_p$ (and thus $\mathcal{H}^2$) as a graph $\Id + \phi$ for a given $\Ad(H)$-equivariant linear map $\phi \colon \mathcal{H}^1_p \to \nn$. 
		
		We are ready now to prove the main result of this section.
	\begin{theorem}[Theorem \ref{thm:B}]\label{thm:connectedness of scal>0}
		Let $M$ be a  stable $U\ltimes_\theta V$-homogeneous manifold. Then the set of $U\ltimes V$-invariant positive scalar curvature metrics on $M$ is contractible.
	\end{theorem}
\begin{proof}
	Let us consider the reductive decomposition $\gg=\mm_\uu\oplus V$. Recall the formula for the scalar curvature $\Rscal(g)$ of a homogeneous Riemannian manifold \cite[Corollary 7.39]{be},
	\[\Rscal(g)= -\frac{1}{2}\tr\Bk_g-\frac{1}{4}\Vert [\cdot,\cdot]\Vert^2_g-\Vert \Hm_g \Vert^2_g,\]
	and let us rewrite it taking into account orthonormal bases $\{V_i\}$ for $V$ and $\{X_i\}$ for $V^\perp$,
	\begin{align*}
		\Rscal = \Rscal_{U/H} -\frac{1}{4} \sum_{i,j,k}g\left(\left[X_i,X_j\right],V_k\right)^2-\frac{1}{2}\sum_i\tr\theta(X_i)^2 -\frac{1}{2}\sum_{i,j,k}g\left(\theta(X_i)V_j,V_k\right)^2-\Vert \Hm_g \Vert^2_g,
	\end{align*}
where $\theta(X_i) \coloneqq \left.\ad X_i\right|_V$. 

Note that if $\Rscal(g) >0$, then $\Rscal^\star (g)>0$, where
\begin{equation}\label{eq:deform scal star}
	\Rscal^\star \coloneqq \Rscal_{U/H}-\frac{1}{4} \sum_{i,j,k}g\left(\left[X_i,X_j\right],V_k\right)^2-\frac{1}{2}\sum_i\tr\theta(X_i)^2-\frac{1}{2}\sum_{i,j,k}g\left(\theta(X_i)V_j,V_k\right)^2
\end{equation}
is the \textit{unimodular scalar curvature}.

We show first that the set of positive unimodular scalar curvature, $\mathcal{M}^G_{upsc}$, is contractible.

Observe that we can interpret this formula in terms of the homogeneous Riemannian submersion $\pi\colon \left(U\ltimes V\right)/H \to U/H$, with the metric $g=g\left(g^B,g^F,\phi\right)$ being constructed from a $U$-homogeneous metric $g^B$ in the base $U/H$, an $\Ad(H)$-invariant inner product $g^F$ on the fiber $V$, and an $\Ad(H)$-equivariant linear map $\phi \colon \mm_\uu \to V$. 

We now note that there is a deformation retract of $\mathcal{M}^G_{upsc}$ to the set of metrics $g$ such that $\Rscal^\star(g) >0$ and $g\left(\mm_\uu,V\right)=0$. Indeed, we just need to take 
\[r_t\left(g\left(g^B,g^F,\phi\right)\right)\coloneqq g\left(g^B,g^F,(1-t)\phi\right). \]
If $\{X_i(g)= U_i + \phi U_i\}$ is the $g$-orthonormal frame obtained by the horizontal lift of a fixed arbitrary $g^B$-orthonormal frame $\{U_i = \pi_*X_i\}$ with respect to the horizontal distribution determined by $\phi$, then let
\[X_i(t)\coloneqq X_i(r_t(g)) = U_i + (1-t)\phi U_i.\]

It is easy to see that the only term in equation \eqref{eq:deform scal star} that changes along $r_t(g)$ is the O'Neill integrability tensor $\sum_{i,j,k}g\left(\left[X_i,X_j\right],V_k\right)^2$, which, in this case where $V$ is abelian, is homogeneous in $(1-t)$ and decreases until it reaches zero at $r_1(g)$. Namely,
\begin{align*}
	r_t(g)\left(\left[X_i(t),X_j(t)\right],V_k\right)^2 =& \ g^F\left(\left[U_i,(1-t)\phi U_j\right]+\left[(1-t)\phi U_i,U_j\right],V_k\right)^2\\
	=& \ (1-t)^2g^F\left(\left[X_i(0),\phi U_j\right]+\left[\phi U_i,X_j(0)\right],V_k\right)^2.
\end{align*}
Thus, $\Rscal^\star(r_t(g))$ is nondecreasing till the metric $r_1(g)$, where the horizontal distribution is given by $\mm_\uu$.

Since $\theta$ is stable, there is a minimal representation $\theta_{\min}$ in the orbit $\GL(V)\cdot \theta$ for the conjugation action of $\GL(V)$ on $\End\left(\uu,\gl(V)\right)$ (see Remark \ref{rem:stable = there is minimum}). Notice that the only critical points for the norm in the orbit $\GL(V)\cdot \theta$ are minima (see \cite[Lemma 1.5.1]{bl20}). Thus, if we consider the equivalence between changing $g^F$ and changing $\theta$ via conjugation, we may deform $\theta$ via the negative moment map flow $\psi$ keeping $\Rscal^\star$ positive. Indeed, the norm of $\theta_t\coloneqq \psi_t\left(\theta\right)$, appearing in formula \eqref{eq:deform scal star} and given by
\[\sum_{i,j,k}g^F_t\left(\theta(X_i)V_j,V_k\right)^2\]
decreases to zero as $\theta_t$ approaches $\SO(V) \cdot \theta_{\min}$. Moreover, since $\theta_t= P_t^{-1}\theta P_t$, for $P_t \in \GL(V)$, then $\tr \theta_t^2 \equiv \tr \theta^2$. Therefore, $\Rscal^\star(g_t)>0$ is preserved along this path. 

Now let $f$ be a representative of the $d^{\text{th}}$-homotopy group of $\{g \in \mathcal{M}^G_{upsc} \ \vert \ g\left(\mm_\uu, V\right)=0\}$ and the path $f_t(x) = \left(g^B(x),\theta_t(x),\mm_\uu\right)$. By compactness of $f(\mathbb{S}^d)$, there is a $T$, such that for all $t \geq T$, $f_t(\mathbb{S}^d)$ is arbitrarily close to $\SO(V)\cdot \theta_{\min}$. By taking a tubular neighborhood of $\SO(V)\cdot \theta_{\min}$, this gives us a homotopy from $f$ to a map $\tilde{f} \colon \mathbb{S}^d \to \SO(V)\cdot \theta_{\min}$, preserving $\Rscal^\star >0$.

By remark \ref{remark:minimal is equivalent to}, we have that the metric defined by $g\left(g^B,\theta_{\min},\mm_\uu \right)$ is a $\theta$-adapted standard metric. By Theorem \ref{thm:main}, the space of $\Rscal^\star>0$ restricted to the space of $\theta$-adapted standard metrics, $\mathfrak{S}_{U\ltimes V}$, is a global attractor. Thus, by Proposition~\ref{prop:contractability of global attract} applied to the unimodular Ricci flow, we conclude that $\mathcal{M}^G_{upsc}\cap\mathfrak{S}_{U\ltimes V}$ is contractible. Therefore, we have a homotopy $\tilde{f}_t \left(\mathbb{S}^d\right) \subset \mathcal{M}^G_{upsc}\cap\mathfrak{S}_{U\ltimes V}$ to a constant function, which finally shows that $\mathcal{M}^G_{upsc}$ is contractible.

By \cite[Lemma 3.5]{bl18}, the unimodular scalar curvature obeys the same evolution equation along the unimodular Ricci flow as the scalar curvature does along the Ricci flow. Hence, by \cite[Theorem 1.1]{laf}, we have that $\mathcal{M}^G_{psc}$ is a global attractor for metrics in $\mathcal{M}^G_{upsc}$. Therefore, again by a simple adaptation of Proposition \ref{prop:contractability of global attract}, we conclude $\mathcal{M}^G_{psc}$ is contractible.
\end{proof}
\begin{remark*}
	It is important to notice that we used the fact that $V$ is abelian when deforming the horizontal distribution to an integrable one. It would be interesting to know if one could refine the argument and find a positive scalar curvature preserving path for the case when $V$ is nilpotent.
\end{remark*}

	We do not really need Theorem \ref{thm:main} in the proof of the theorem above. Indeed, in our setting we assume that $U$ has a compact Lie algebra, so its semisimple factor is compact. Thus, by \cite[Proposition 3.9]{jp}, the condition $\theta=\theta_{\min}$ implies that $\kk$ acts skew-symmetrically on $V$ and that $\hat{\zz}$ acts via operators whose transposes commute with $\theta(\uu)$ (see \cite[Lemma 3.6]{jp}). These properties are independent of the metric restricted to $\mm_\uu$. Moreover, the condition $\theta=\theta_{\min}$, together with $\mm_\uu \perp_g V$, implies that $\left.\ric\right|_{V\times V}\equiv 0$ (see Remark \ref{remark:minimal is equivalent to}). Hence, these conditions combined are invariant under the Ricci flow. Finally, using the formula in Remark \ref{remark:Ricci_U in terms of intrinsic and V-action} for the Ricci tensor in the $T_pK/H$ direction, we immediately obtain that the Ricci flow in this case has finite extinction time.
	\begin{remark*}\label{remark:general GIT-stable}
			In general, one could define stable $U\ltimes_\theta N$-homogeneous manifolds for $\gg = \uu\ltimes_\theta \nn$, with $\nn$ a nilpotent ideal and $\uu$ reductive (not necessarily compact), where $\zz(\uu)$ acts on $\nn$ via semisimple operators. In this setting, there exist inner products on $\uu$ and $\nn$ such that $\theta$ is minimal (see \cite{mos} and \cite[Lemma 3.1]{jp}). However, note that the existence of a minimum for the conjugation action of $\GL(V)$ on $\End\left(\uu,\gl(V)\right)$ may depend on the metric $\left.g\right|_{\uu\times\uu}$. As a consequence, the dynamical setting we are investigating becomes more complicated in this more general case.
	\end{remark*}
	
	We conclude this section by proving the following observation on the unimodular Ricci flow invariance of a particularly nice set of $U\ltimes_\theta N$-invariant metrics, for the case where $U\ltimes_\theta N$ is stable in the sense of Remark \ref{remark:general GIT-stable}.
	\begin{proposition}\label{prop:nilsoliton+closed under transpose is RF inv}
		Let $M$ be an $U\ltimes_\theta N$-homogeneous manifold, where $U$ is a reductive Lie group and $N$ the nilradical, with reductive decomposition $\mm_\uu\oplus\nn$, $\mm_\uu \subset \uu$. Then the set of $U\ltimes_\theta N$-invariant Riemannian metrics $g$ such that $g\left(\mm_\uu,\nn\right)=0$, $\left(N,\left.g\right|_N\right)$ is a nilsoliton, and $\theta(\uu)^{t_g} \subset \Der\left(\nn\right)$ is Ricci flow invariant.
	\end{proposition}
\begin{proof}
	Let $X \in \mm_{\uu}$ and $A \in \nn$ and let us consider the mixed Ricci term (see equation \eqref{eq:ricci XY formula})
	\begin{align*}
		\ric_g\left(X,A\right) =& -\frac{1}{2}\Bk\left(X,A\right) +\mo_g\left(X,A\right)-\h_g(X,A)\\
		=&\mo_g\left(X,A\right)\\
		=&-\frac{1}{2}\sum_i g\left([X,A_i]_{\mm},[A,A_i]\right).
	\end{align*}
Where we used in the second equality that the nilradical is in the kernel of the Killing form $\Bk$, thus $\Bk(X,A)=0$; and that $\Hm_g \in \mm_\uu$, thus $\h_g(X,A)=0$. And in the third equality that $\uu$ is a subalgebra, and that $\nn$ is an ideal.

Therefore,
	\begin{align*}
		\ric_g\left(X,A\right) =&-\frac{1}{2}\sum_i g\left([X,A_i]_{\mm},[A,A_i]\right)\\
	=& \tr \left(\theta\left(X\right)^{t_g}\circ\ad\left(A\right)\right)=0.
\end{align*}
This holds because we assumed that $\theta\left(X\right)^{t_g}$ is a derivation, and the inner derivations form an ideal on the full space of derivations. Consequently, we have that $\theta\left(X\right)^{t_g}\circ\ad\left(A\right)$ is a nilpotent operator (see the argument in the proof of Proposition \ref{prop:Ricci invariant semidirect}).

Now we need to show that the conditions of being a nilsoliton and of $\theta(\uu)^{t_g} \subset \Der\left(\nn\right)$ are also preserved. 

Indeed, by \cite[Proposition 1.1]{lau01}, the nilsoliton metric $\left.g\right\vert_N$ is an algebraic Ricci soliton. That is, there exists a $\left.g\right\vert_{\nn\times \nn}$-symmetric derivation $D \in \Der(\nn)$ such that 
\begin{equation}
	\Ric_N =  c\Id + D
\end{equation}
for a constant $c$. Thus, we have that
\begin{align*}
	\left.\Ric_g\right\vert_\nn =& \  \Ric_N + \frac{1}{2}\sum_k\left[\theta\left(U_k\right),\theta\left(U_k\right)^{t_g}\right]+S^g(\ad H_g)
	= c\Id + \tilde{D}_g,
\end{align*}
where $\tilde{D}_g=\frac{1}{2}\sum_k\left[\theta\left(U_k\right),\theta\left(U_k\right)^{t_g}\right]+S^g(\ad H_g)+D$ is an $\Ad(H)$-equivariant symmetric derivation of $\nn$, since each independent term is. Therefore, the derivative of the metric, $\left.\dot{g}\right\vert_N$, is tangent to the space of nilsoliton metrics.

Finally, observe that if $\left.g\right\vert_{\nn\times\nn}=\langle P \cdot,\cdot\rangle$ for a positive definite operator $P$ and a given background metric $\langle\cdot,\cdot\rangle$, then $\theta^{t_g} = P^{-1}\theta^tP$, where this last transpose is taken with respect to $\langle\cdot,\cdot\rangle$. We then have that
\begin{align*}
	\dot{\theta}^{t_g} =& \ \left[\theta^{t_g}, \left.\Ric_g\right\vert_\nn\right] =  \left[\theta^{t_g}, \tilde{D}_g\right] \in \Der(\nn).
\end{align*}
\end{proof}

\begin{remark*}
	It follows immediately from the proof of Proposition \ref{prop:nilsoliton+closed under transpose is RF inv} that the conditions on $g$ above are also invariant under the unimodular Ricci flow. 
\end{remark*}

\end{document}